\documentclass[12pt]{report}

\usepackage{amsmath,amssymb,amsfonts,dsfont,mathrsfs}

\usepackage{graphicx}
\usepackage{import}
\usepackage{pdflscape}
\usepackage{array}
\usepackage{longtable}
\usepackage{scalerel}
\usepackage{caption}
\usepackage{subcaption}
\usepackage{float}
\usepackage{commath}
\usepackage{mathtools}
\usepackage{empheq}
\usepackage{pdfcomment}
\usepackage{indentfirst}
\usepackage{tikz}
\usetikzlibrary{arrows,positioning,calc}
\usetikzlibrary{fit}
\usepackage{enumitem}

\usepackage[a4paper,margin=2cm]{geometry}

\usepackage{amsthm}

\newtheorem{thrm}{Theorem}[section]
\newtheorem{lemm}[thrm]{Lemma}
\newtheorem{prop}[thrm]{Proposition}

\newtheorem{defin}{Definition} 

\theoremstyle{definition}
\newtheorem*{ex}{Example}
\newtheorem*{remark}{Remark}

\usepackage{xpatch}
\xpatchcmd{\proof}{#1}{#1}{}{}

\makeatletter
\newcommand*{\rom}[1]{\expandafter\@slowromancap\romannumeral #1@}
\makeatother

\DeclareMathOperator{\End}{End}
\numberwithin{equation}{section}

\let\leq\leqslant

\let\geq\geqslant
\let\alb\allowbreak

\newcommand{\theArticleName}{Article name}

\newcommand{\ArticleName}[1]{\renewcommand{\theArticleName}{#1}\vspace{-7mm}\par\noindent {\LARGE\bf #1\par}}
\newcommand{\Author}[1]{\vspace{5mm}\par\noindent {\it #1} \par\vspace{2mm}\par}
\newcommand{\Address}[1]{\vspace{2mm}\par\noindent {\it #1} \par}
\newcommand{\Email}[1]{\ifthenelse{\equal{#1}{}}{}{\par\noindent {\rm E-mail: }{\it #1} \par}}
\newcommand{\EmailMarked}[1]{\ifthenelse{\equal{#1}{}}{}{\par\noindent $^*$~{\rm E-mail: }{\it #1} \par}}
\newcommand{\URLaddress}[1]{\ifthenelse{\equal{#1}{}}{}{\par\noindent {\rm URL: }{\tt #1} \par}}
\newcommand{\URLaddressMarked}[1]{\ifthenelse{\equal{#1}{}}{}{\par\noindent $^*$~{\rm URL: }{\tt #1} \par}}
\newcommand{\EmailD}[1]{\ifthenelse{\equal{#1}{}}{}{\par\noindent {$\phantom{\dag}$~\rm E-mail: }{\it #1} \par}}
\newcommand{\EmailDD}[1]{\ifthenelse{\equal{#1}{}}{}{\par\noindent {$\phantom{{}^{\dag^1}}$~\rm E-mail: }{\it #1} \par}}
\newcommand{\URLaddressD}[1]{\ifthenelse{\equal{#1}{}}{}{\par\noindent {$\phantom{\dag}$~\rm URL: }{\tt #1} \par}}
\newcommand{\URLaddressDD}[1]{\ifthenelse{\equal{#1}{}}{}{\par\noindent {$\phantom{\dag^1}$~\rm URL: }{\tt #1} \par}}

\mathtoolsset{showonlyrefs}
\begin{document}

\ArticleName{New combinatorial formulae for nested Bethe\\vectors II} 

\Author{Maksim KOSMAKOV\;${\vphantom1}^\circ$ and Vitaly TARASOV\;${\vphantom1}^*$}

\Address{$^\circ$Department of Mathematical Sciences, University of Cincinnati,\\
\hphantom{$^*$}P.O.~Box 210025, Cincinnati, OH 45221, USA}
\EmailD{\href{mailto: kosmakmm@ucmail.uc.edu}{kosmakmm@ucmail.uc.edu}}

\Address{$^*$Department of Mathematical Sciences,
Indiana University Indianapolis,\\ 
\hphantom{$^*$}402 North Blackford St, Indianapolis, IN 46202-3216, USA}
\EmailD{\href{mailto:vtarasov@iu.edu}{vtarasov@iu.edu}} 

\section*{Abstract}
We give new combinatorial formulae for vector-valued weight functions (off-shell
nested Bethe vectors) for the evaluation modules over the Yangian $Y(\mathfrak{gl}_n)$.

\section{Introduction}
In this paper we construct new combinatorial formulae for vector-valued weight functions for
evaluation modules over the Yangian $Y(\mathfrak{gl}_n)$. The weight functions, otherwise called
(off-shell) nested Bethe vectors, are quite important in the theory of quantum integrable
models and representation theory of Lie algebras and quantum groups. Their initial definition
originates in the framework of the nested algebraic Bethe ansatz, an ingenious technique to find
eigenvectors and eigenvalues of transfer matrices of lattice integrable models associated with
higher rank Lie algebras \cite{KulRes82,KulRes83}. An excellent review of the algebraic Bethe
ansatz can be found in \cite{S1, S2}. The results of \cite{KulRes83} has been extended
to higher transfer matrices in \cite{MTV1}.

On the other hand, the vector-valued weight functions is an essential part of construction
of hypergeometric solutions to the quantized (difference) Knizhnik-Zamolodchikov equations
\cite{TV, MTT}. They also appeared in various related problems \cite{KPT, TV3, TV4, MTV2}.
In the last decade the weight functions were connected to the stable envelopes for cotangent
bundles of partial flag varieties that are particular examples of Nakajima quiver varieties,
see \cite{RTV1, RTV2, RTV3, RTV4, TV5}.

For various problems, it is desirable to have sufficeintly explicit expressions for vector-valued
weight functions for tensor products of evaluation modules over $Y(\mathfrak{gl}_n)$. Such
expressions can be constructed in two steps. The first one is to deal with a single evaluation
module and the second one is to merge expressions for individual evaluation modules into
an expression for the whole tensor product. In this paper, we will consider the first step
of this construction. The second step is fairly standard, see Theorem~3.5 in \cite{TVC}.

An evaluation $Y(\mathfrak{gl}_n)$-module is the pull-back of a $\mathfrak{gl}_n$-module
via the evaluation homomorphism $Y(\mathfrak{gl}_n)\rightarrow U(\mathfrak{gl}_n)$, see
Section \ref{notation}. The goal in general is to expand the vector-valued weight function
for the evaluation $Y(\mathfrak{gl}_n)$-module in some natural basis of the underlying
$\mathfrak{gl}_n$-module and find expressions for the coordinates. For Verma modules over
$\mathfrak{gl}_n$, such kind of expressions were written in \cite{TVC}. In this paper,
we give a generalization of formulae from \cite{TVC}.

Combinatorial formulae for the vector-valued weight functions associated with the differential
Knizhnik-Zamolodchikov equations were developed in \cite{M,SV1,SV2,RSV,FRV,MV}.

The expressions for weight functions in \cite{TVC} are based on recursions
induced by the standard embeddings of Lie algebras,
$\,\mathfrak{gl}_1\oplus \mathfrak{gl}_{n-1}\subset \mathfrak{gl}_n\,$ and
$\;\mathfrak{gl}_{n-1}\oplus \mathfrak{gl}_1 \subset \mathfrak{gl}_n$. The recursions
allow one to write down weight functions for $Y(\mathfrak{gl}_n)$ via weight functions for
$Y(\mathfrak{gl}_{n-1})$. This results in formulae for coordinates of weight functions in bases
of Verma $\mathfrak{gl}_n$-modules of the form
\begin{equation}\label{basis} \Bigl\{ \prod_{i>j} e_{ij}^{m_{ij}}v,
\quad m_{ij}\in\mathbb{Z}_{\geq0}\Bigr\},
\end{equation}
where $e_{ij}$ are the standard generators of $\mathfrak{gl}_n$, $\;v$ is the highest weight
vector, and a certain ordering of noncommuting factors $e_{ij}^{m_{ij}}$ in the product is
imposed. The ordering is determined by a chain of embeddings
$\;\mathfrak{gl}_1 \oplus\dots\oplus \mathfrak{gl}_1\subset\dots\subset\mathfrak{gl}_n$,
where the $k$-th element of the chain is induced either by the embedding
$\,\mathfrak{gl}_1\oplus \mathfrak{gl}_{k-1}\subset \mathfrak{gl}_k\,$ or by the embedding
$\;\mathfrak{gl}_{k-1}\oplus \mathfrak{gl}_1 \subset \mathfrak{gl}_k$. For instance, the chain
\begin{itemize}
\item[]\strut\kern-1em
$\mathfrak{gl}_1 \oplus\dots\oplus \mathfrak{gl}_1\subset
\mathfrak{gl}_2 \oplus\mathfrak{gl}_1 \oplus\dots\oplus \mathfrak{gl}_1\subset
\dots\subset\mathfrak{gl}_{n-2}\oplus \mathfrak{gl}_1 \oplus \mathfrak{gl}_1 \subset
\mathfrak{gl}_{n-1}\oplus \mathfrak{gl}_1 \subset \mathfrak{gl}_n$
\end{itemize}
yields the orderings
\begin{itemize}
\item[]\strut\kern-1em
$e_{ij}^{m_{ij}}$ is to the left of $e_{kl}^{m_{kl}}$ if $i>k$ or $i=k$, $j>l$,
\end{itemize}
while the chain
\begin{itemize}
\item[]\strut\kern-1em
$\mathfrak{gl}_1 \oplus\dots\oplus \mathfrak{gl}_1\subset
\mathfrak{gl}_1 \oplus\dots\oplus \mathfrak{gl}_1\oplus \mathfrak{gl}_2\subset
\dots\subset\mathfrak{gl}_1 \oplus \mathfrak{gl}_1\oplus\mathfrak{gl}_{n-2} \subset
\mathfrak{gl}_1\oplus \mathfrak{gl}_{n-1}\subset \mathfrak{gl}_n$
\end{itemize}
yields the orderings
\begin{itemize}
\item[]\strut\kern-1em
$e_{ij}^{m_{ij}}$ is to the left of $e_{kl}^{m_{kl}}$ if $j<l$ or $j=l$, $i<k$.
\end{itemize}
These two orderings are examples of normal orderings of the factors $e_{ij}^{m_{ij}}$
corresponding to normal orderings of positive roots of the Lie algebra $\,\mathfrak{gl}_n$.
An ordering is called normal is for any $i>j>k$, the factor $e_{ik}^{m_{ik}}$ is located
between the factors $e_{ij}^{m_{ij}}$ and $e_{jk}^{m_{jk}}$.

Formulae established in \cite{TVC} do not cover all normal orderings of noncommuting factors
in \eqref{basis}. The first example occurs at $n=4$ and looks as follows,
\begin{equation}\label{basis4}
\Bigl\{e_{32}^{m_{32}}e_{31}^{m_{31}}e_{42}^{m_{42}}e_{41}^{m_{41}}
e_{21}^{m_{21}}e_{43}^{m_{43}}v,\qquad m_{ij}\in\mathbb{Z}_{\geq0}\Bigr\}.
\end{equation}
This example and similar ones for $n>4$ are important in applications,
see for instance \cite{MV}.

To enlarge the set of covered orderings, we consider recursions based on embeddings
\begin{equation}\label{newembed}
\mathfrak{gl}_m\oplus\mathfrak{gl}_{n-m}\subset \mathfrak{gl}_n \qquad \text{with} \;\; 1<m<n-1.
\end{equation}
For instance, the embedding $\mathfrak{gl}_2\oplus \mathfrak{gl}_2\subset \mathfrak{gl}_4$ yields
example \eqref{basis4}. We worked out example \eqref{basis4} in detail in our previous
paper \cite{KT}. In this paper we consider the general case. The reasoning is very similar
to the $\,\mathfrak{gl}_4$ case and we suggest taking a look into \cite{KT} to clarify
technicalities if needed.

The main result of the paper is Theorem \ref{mainth}. It gives an expression for
the weight function for $Y(\mathfrak{gl}_n)$ via the weight functions for $Y(\mathfrak{gl}_m)$
and $Y(\mathfrak{gl}_{n-m})$. For $m=1$ and $m=n-1$, Theorem \ref{mainth} becomes respectively
Theorem~3.3 and Theorem~3.1 in \cite{TV}.

Unlike \cite{TVC}, we will consider only the case of weight functions for Yangian modules
(the rational case). The weight functions for modules over the quantum loop algebra
$U_q(\widetilde{\mathfrak{gl}_n})$, the trigonometric case, cannot be dealt with
by our current approach because of essential noncommutativity of $q$-analogues of the generators
$e_{ij}$, $\;i>j$. For instance, the obstacle in example \eqref{basis4} comes from the relation
\begin{equation}
e_{42}e_{31}-e_{31}e_{42}=(q-q^{-1})e_{32}e_{41}
\end{equation}
that holds in the trigonometric case. This kind of obstruction does not show up in \cite{TVC},
but prevents one from extending the reasoning used in this paper to the trigonometric case.

An alternative approach to get explicit expressions for the vector-valued weight functions
in the trigonometric case was developed in \cite{KPT, KP, OPS}. It is based on considering
composed currents and half-currents in the quantum affine algebra and their projections
on two Borel subalgebras of different kind. Although currents and half-currents are rather
complicated expressions in the original generators of the quantum affine algebra,
their commutation relations turn out to have a simple form in certain cases. This approach
allows one to recover combinatorial expression for vector-valued weight functions in evaluation
modules in the trigonometric case obtained in \cite{TVC}. It is an interesting question to obtain
trigonometric analogues of new combinatorial expressions for vector-valued weight functions
developed in this paper within the composed currents approach. This might require working out
commutation relations of composed currents with interlacing indices, like $F_{3,1}(t)$ and
$F_{4,2}(t)$ in the notation of \cite{KPT}.

\paragraph{Acknowledgments:}The second author is supported in part by Simons Foundation grants
430235, 852996. The authors thank an anonymous referee for valuable remarks.

\section{Notations}\label{notation}
We will be using the standard superscript notation for embeddings of tensor factors into tensor products. For a tensor product of vector spaces $V_1\otimes V_2\otimes\ldots\otimes V_k$ and an operator $A\in \End (V_i)$, denote
\begin{equation}
A^{(i)}=1^{\otimes(i-1)} \otimes A \otimes 1^{\otimes(k-i)} \in \End(V_1\otimes V_2\otimes\ldots\otimes V_k). \end{equation}
Also, if $B\in \End (V_j)$, $i\neq j$, denote $(A\otimes B)^{(ij)}=A^{(i)}B^{(j)}$, etc.

Fix a positive integer $n$. All over the paper we identify elements of End $\mathbb{C}^n$ with $n\times n$ matrices using the standard basis of $\mathbb{C}^n$. That is, for $L\in \End \mathbb{C}^n$ we have $ \ L=\left( L_b^a \right)_{a,b=1}^n$, where $L_b^a$ are the entries of $L$. Entries of matrices acting in the tensor products $\left(\mathbb{C}^n\right)^{\otimes k}$ are naturally labeled by multiindices. For instance, if $M\in \End \left(\mathbb{C}^n\otimes \mathbb{C}^n\right)$, then $M=\left( M_{cd}^{ab} \right)_{a,b,c,d=1}^n$.

The rational $R$-matrix is $R(u)\in \End (\mathbb{C}^n\otimes \mathbb{C}^n)$,

\begin{equation}\label{rmatrix}
R(u)=1+\dfrac{1}{u}\sum_{a, b=1}^{n} E_{a b} \otimes E_{b a},
\end{equation}
where $ E_{a b} \in \End\left(\mathbb{C}^{n}\right) $ is the matrix with the only nonzero entry equal to 1 at the intersection of the $a$-th row and $b$-th column. The entries of $R(u)$ are
\begin{equation}\label{rentry}
R^{ab}_{cd}(u)= \delta_{ac}\delta_{bd} +\dfrac{1}{u}\delta_{ad}\delta_{bc}.
\end{equation}
The $R$-matrix satisfies the Yang-Baxter equation
\begin{equation}\label{YB} R^{(12)}(u-v) R^{(13)}(u) R^{(23)}(v)=R^{(23)}(v) R^{(13)}(u) R^{(12)}(u-v).\end{equation}

The Yangian $Y(\mathfrak{g l}_{n})$ is a unital associative algebra with generators $ \left(T^a_b\right)^{\{s\}}, a, b=1, \ldots, n, $ and $s=1,2, \ldots$. Organize them into generating series:

\begin{equation}\label{toperator} T^a_b(u)=\delta_{ab}+\sum_{s=1}^{\infty} \left(T^a_b\right)^{\{s\}} u^{-s}, \quad a, b=1, \ldots, n .\end{equation}
The defining relations in $Y(\mathfrak{g l}_{n})$ are
\begin{equation}\label{trel} (u-v)\left[T^a_b(u), T^c_d(v)\right]=T^a_d(u)T^c_b(v) -T^a_d(v)T^c_b(u) \end{equation}
for all $a, b, c, d=1, \ldots, n.$

Combine series $\eqref{toperator}$ into a matrix $T(u)=\sum\limits_{a, b=1}^{n} E_{a b} \otimes T_{ b}^a(u)$ with entries in $ Y(\mathfrak{g l}_{n})$. Then relations \eqref{trel} amount to the following equality
\begin{equation*}
R^{(12)}(u-v) T^{(1)}(u) T^{(2)}(v)=T^{(2)}(v) T^{(1)}(u) R^{(12)}(u-v) ,
\end{equation*}
where $T^{(1)}(u)=\sum\limits_{a, b=1}^{n} E_{a b}\otimes 1 \otimes T_{ b}^a(u)$ and $T^{(2)}(v)=\sum\limits_{a, b=1}^{n} 1 \otimes E_{a b}\otimes T_{ b}^a(v)$.

The Yangian $Y(\mathfrak{g l}_{n})$ is a Hopf algebra. In terms of generating series $\eqref{toperator}$, the coproduct $\Delta: Y(\mathfrak{g l}_{n}) \rightarrow Y(\mathfrak{g l}_{n}) \otimes Y(\mathfrak{g l}_{n})$ reads as follows:
\begin{equation}\label{coproduct}
\Delta\left(T_{ b}^a(u)\right)=\sum_{c=1}^{n} T^c_{ b}(u) \otimes T^a_{ c}(u), \quad a, b=1, \ldots, n .
\end{equation}
Denote by $\widetilde{\Delta}:Y(\mathfrak{g l}_{n}) \rightarrow Y(\mathfrak{g l}_{n}) \otimes Y(\mathfrak{g l}_{n})$ the opposite coproduct
\begin{equation}\label{opcoproduct}
\widetilde{\Delta}\left(T_{ b}^a(u)\right)=\sum_{c=1}^nT^a_c(u)\otimes T^c_b(u), \quad a, b=1, \ldots, n .
\end{equation}

Fix a collection of nonnegative integers $\xi_1,\xi_2,\ldots,\xi_{n-1}$. Set $\boldsymbol{\xi}=(\xi_1,\xi_2,\ldots,\xi_{n-1})$ and $\xi^a=\xi_1+\ldots+\xi_a$, $a=1,\ldots,n-1$.
Consider the variables $t^a_i$,\; $a=1,\ldots, n-1$,\; $i=1,\ldots \xi_a$. We will write
\begin{equation}\label{tvar}
\boldsymbol{t}^a=(t^a_1,\ldots, t^a_{\xi_a}),\qquad \boldsymbol{t}=(\boldsymbol{t}^1,\ldots,\boldsymbol{t}^{n-1}).
\end{equation}

We will use the ordered product notation for any noncommuting factors $X_1,\ldots,X_k$,
\begin{equation}
\overset{\rightarrow}{\prod\limits_{1\leq i\leq k}} X_i= X_1X_2\ldots X_k, \quad \quad \overset{\leftarrow}{\prod\limits_{1\leq i\leq k}} X_i= X_kX_{k-1}\ldots X_1.
\end{equation}
Consider the vector space $ (\mathbb{C}^n)^{\otimes \xi^{n-1}}$ and define
\begin{equation}\label{tbj}
\overset{[j]}{\mathbb{T}}({\boldsymbol{t^j}})=\overset{\rightarrow}{\prod\limits_{ 1\leq k\leq \xi_j}} {T}^{(\xi^{j-1}+k)}(t_k^j),\quad\quad \overset{[k,j]}{\mathbb{R}}(\boldsymbol{t^k},\boldsymbol{t^j})=
\overset{\rightarrow}{\prod\limits_{1\leq i\leq\xi_k}}\Biggl(\;
\overset{\leftarrow}{\prod\limits_{1\leq l\leq\xi_j}} {R}^{(\xi^{k-1}+i,\xi^{j-1}+l)}(t^k_i-t^j_l)\Biggr),
\end{equation}
where we view ${T}^{(\xi^{j-1}+k)}(t_k^j)$\ as a matrix with entries in $Y(\mathfrak{gl}_{n})$ acting on $(\xi^{j-1}+k)$-th copy of $\mathbb{C}^n$ in $ (\mathbb{C}^n)^{\otimes \xi^{n-1}}$. For the expression
\begin{equation}\label{boldt}
\widehat{\mathbb{T}}_{\boldsymbol\xi}(\boldsymbol{t})=\ \overset{[1]}{\mathbb{T}}(\boldsymbol{t}^1)\ldots \overset{[n-1]}{\mathbb{T}}(\boldsymbol{t}^{n-1})\overset{\leftarrow}{\prod\limits_{1\leq i\leq n-1}}\Biggl(\;\overset{\leftarrow}{\prod\limits_{1\leq j< i}}\overset{[i,j]}{\mathbb{R}}(\boldsymbol{t^i},\boldsymbol{t^j})\Biggr),
\end{equation}
denote by ${\mathbb{B}}_{\boldsymbol\xi}(\boldsymbol{t})$ the following entry
\begin{equation} \label{bethevect}
{\mathbb{B}}_{\boldsymbol\xi}(\boldsymbol{t})=\left( \widehat{\mathbb{T}}_{\boldsymbol\xi}(\boldsymbol{t}) \right)^{\boldsymbol{1}^{\xi_1}, \,\boldsymbol{2}^{\xi_2},\,\ldots,\,\boldsymbol{n-1}^{\xi_{n-1}}}_{\boldsymbol{2}^{\xi_1},\,\boldsymbol{3}^{\xi_2},\,\ldots,\,\boldsymbol{n}^{\xi_{n-1}}},
\end{equation}
where
\begin{equation}
\begin{aligned}
&\boldsymbol{1}^{\xi_1}, \boldsymbol{2}^{\xi_2}, \ldots, \boldsymbol{(n-1)}^{\xi_{n-1}} =\underbrace{ 1, 1,\ldots 1}_{\xi_1},\;\underbrace{ 2,2,\ldots2}_{\xi_2},\;\ldots \;,\underbrace{n-1,\;n-1,\;\ldots,\;n-1}_{\xi_{n-1}}\;,\\
&\boldsymbol{2}^{\xi_1}, \boldsymbol{3}^{\xi_2}, \ldots, \boldsymbol{n}^{\xi_{n-1}} =\underbrace{ 2, 2,\ldots, 2}_{\xi_1},\;\underbrace{ 3,3,\ldots,3}_{\xi_2},\;\ldots, \;\underbrace{n,\, n,\,\ldots,\, n}_{\xi_{n-1}}\;.
\end{aligned}
\end{equation}
To indicate the dependence on $n$, if necessary, we will write $\mathbb{B}_{\boldsymbol\xi}^{\langle n\rangle}(\boldsymbol{t})$ .

There is a one-parameter family of automorphisms $\rho_x$ : $Y(\mathfrak{gl}_n) \rightarrow Y(\mathfrak{gl}_n)$ defined in terms of the series $T(u)$ by the rule $\rho_x T(u)= T(u-x)$, where in the right-hand side, each expression $(u-x)^{-s}$ has to be expanded as a power series in $u^{-1}$.

Denote by $e_{ab}$, $a$,$b = 1,\ldots,n$, the standard generators of the Lie algebra $\mathfrak{gl}_n$. A vector $v$ in a $\mathfrak{gl}_n$ -module is called
singular of weight $(\Lambda^1,\ldots,\Lambda^n)$ if $e_{ab}v = 0$ for all $a<b$ and $e_{aa}v=\Lambda_av$ for all $a=1,\ldots,n$.

The Yangian $Y(\mathfrak{g l}_{n})$ contains the universal enveloping algebra $U(\mathfrak{g l}_{n})$ as a Hopf subalgebra. The embedding is given by the rule $e_{a b} \mapsto (T_{ a}^b)^{\{1\}}$ for all $a, b=1, \ldots, n$. We identify $U(\mathfrak{g l}_{n})$ with its image in $Y(\mathfrak{g l}_{n})$ under this embedding.

The evaluation homomorphism $\epsilon$ : $Y (\mathfrak{gl}_n) \rightarrow U(\mathfrak{gl}_n)$ is given by the rule $\epsilon$ : $(T_{b}^a)(u) \mapsto \delta_{ab}+e_{ba}u^{-1}$ for all $a,b = 1,...,n$. Both the automorphisms $\rho_x$ and the homomorphism $\epsilon$ restricted to the subalgebra $U(\mathfrak{gl}_n)$ are the identity maps.

For a $\mathfrak{gl}_n$-module $V$ denote by $V(x)$ the $Y(\mathfrak{gl}_n)$-module induced from $V$ by the homomorphism $\epsilon\circ \rho_x$. The module $V(x)$ is called an evaluation module over $Y(\mathfrak{gl}_n)$.

A vector $v$ in a $Y(\mathfrak{gl}_n)$-module is called singular with respect to the action of $Y(\mathfrak{gl}_n)$ if $T^{a}_{b}(u)v = 0$ for all $1\leq b < a \leq n$. A singular vector $v$ that is an eigenvector for the action of $T_{1}^{1}(u)$,...,$T_{n}^{n}(u)$ is called a weight singular vector, and the respective eigenvalues are denoted by $\langle T^1_{1}(u)v\rangle,\ldots, \langle T^n_{n}(u)v\rangle$.

\begin{ex}\label{evaluation} Let $V$ be a $\mathfrak{g l}_{n}$-module and $v \in V$ be a $\mathfrak{gl}_n$-singular vector of weight $\left(\Lambda^{1}, \ldots, \Lambda^{n}\right)$. Then $v$ is a weight singular vector with respect to the action of \,$Y(\mathfrak{g l}_{n})$ in the evaluation module $V(x)$ and $\left\langle T^a_{ a}(u) v\right\rangle=1+\Lambda^{a}(u-x)^{-1}$, $a=1, \ldots, n$.
\end{ex}

For $k<n$ we consider two embeddings of the algebra $Y(\mathfrak{gl}_{k})$ into $Y(\mathfrak{gl}_{n})$, called $\phi_k$ and $\psi_k$:
\begin{equation}\label{embeddings}
\phi_k(T^{\langle k\rangle}(u))^a_b=(T^{\langle n\rangle}(u))^a_b \hspace{2cm} \psi_k(T^{\langle k\rangle}(u))^a_b=(T^{\langle n\rangle}(u))^{a+n-k}_{b+n-k}(u)
\end{equation}
$a,b=1,\ldots, k$. Here $(T^{\langle k\rangle}(u))^a_b$ and $((T^{\langle n\rangle}(u))^a_b$ are the series $T^a_b(u)$ for the algebras $Y(\mathfrak{gl}_{k})$ and $Y(\mathfrak{gl}_{n})$, respectively.

\section{Splitting property}

Let $T_{a b}^{\langle r\rangle}(u)$ be series ($\ref{toperator}$) for the algebra $Y(\mathfrak{g l}_{r})$, and $R^{\langle r\rangle}(u)$ be the corresponding rational $R$-matrix, see $\eqref{rmatrix}$. For the rest of the paper we fix integers $m$ and $n$, $1\leq m<n$.

Consider $Y(\mathfrak{g l}_{n-m})$-module structure on the vector space $\mathbb{C}^{n-m}$
given by the rule
\begin{equation}\label{lmodule}
\pi(x): T^{\langle n-m \rangle}(u) \mapsto(u-x)^{-1} R^{\langle n-m\rangle}(u-x).
\end{equation}
Let $ \mathrm{\mathbf v}_{ 1}, \mathrm{\mathbf v}_{ 2},\ldots, \mathrm{\mathbf v}_{ n-m} $ be the standard basis of the space $\mathbb{C}^{n-m}$. The $Y(\mathfrak{g l}_{n-m})$-module defined by \eqref{lmodule} is a highest weight evaluation module with $\mathfrak{g l}_{n-m}$ highest weight $(1,\ldots,0,0)$ and highest weight vector $ \mathrm{\mathbf v}_{ 1}$.

Consider $Y(\mathfrak{g l}_{m})$-module structure on the vector space $\mathbb{C}^{m}$
given by the rule
\begin{equation}\label{lbmodule}
\varpi(x): T^{\langle m\rangle}(u) \mapsto(x-u)^{-1}\Bigl(\Bigl(R^{\langle m\rangle}(x-u)\Bigr)^{(21)}\Bigr)^{t_{2}},
\end{equation}
where the superscript $t_{2}$ stands for the matrix transposition in the second tensor factor.

We denote the standard basis of the space $\mathbb{C}^{m}$ by $\mathbf{w}_{1}, \mathbf{w}_{ 2},\ldots, \mathbf{w}_{m} $. The $Y(\mathfrak{g l}_{m})$-module defined by \eqref{lbmodule} is a highest weight evaluation module with $\mathfrak{g l}_{m}$ highest weight $(0,\ldots,0,-1)$ and highest weight vector $\mathbf{w}_{m}$.

For any $Z\in$ $\operatorname{End}(\mathbb{C}^{n-m})$, set $ {\nu}(Z)=Z \mathrm{\mathbf v}_{ 1}$, and for any $X \in$ $\operatorname{End}(\mathbb{C}^{m})$, set $\bar{\nu}(X)=X \mathbf{w}_{m}$.
Recall the coproducts $\Delta$ and $\widetilde{\Delta}$, see \eqref{coproduct} and \eqref{opcoproduct}, and the embeddings $\psi_{n-m}: Y(\mathfrak{g l}_{n-m})\rightarrow Y(\mathfrak{g l}_{n})$, $\phi_{m}: Y(\mathfrak{g l}_{m})\rightarrow Y(\mathfrak{g l}_{n})$ given by \eqref{embeddings}. For any $r$, denote by $ (\Delta^{\langle r\rangle})^{(k)}: Y(\mathfrak{g l}_{r}) \rightarrow\left(Y(\mathfrak{g l}_{r})\right)^{\otimes(k+1)}$ and $ (\widetilde{\Delta}^{\langle r\rangle})^{(k)}: Y(\mathfrak{g l}_{r}) \rightarrow\left(Y(\mathfrak{g l}_{r})\right)^{\otimes(k+1)}$ the corresponding iterated coproduct and opposite coproduct.
Consider the maps
\begin{equation}
\psi_{n-m}\left(x_{1}, \ldots, x_{k}\right): Y(\mathfrak{g l}_{n-m}) \rightarrow\left(\mathbb{C}^{n-m}\right)^{\otimes k} \otimes Y(\mathfrak{g l}_{n}),\end{equation}
\begin{equation}
\psi_{n-m}(x_{1}, \ldots, x_{k})= ( {\nu}^{\otimes k} \otimes \mathrm{id}) \circ\left(\pi\left(x_{1}\right) \otimes \cdots \otimes \pi\left(x_{k}\right) \otimes \psi_{n-m}\right) \circ(\Delta^{\langle n-m\rangle})^{(k)},
\end{equation}
and
\begin{equation}\phi_{m}\left(x_{1}, \ldots, x_{k}\right): Y(\mathfrak{g l}_{m}) \rightarrow \left(\mathbb{C}^{m}\right)^{\otimes k} \otimes Y(\mathfrak{g l}_{n}) ,\end{equation}
\begin{equation}
\phi_{m}\left(x_{ 1}, \ldots, x_{k}\right)= (\bar{\nu}^{\otimes k} \otimes \mathrm{id}) \circ\left( \varpi\left(x_{ 1}\right) \otimes \cdots \otimes \varpi \left(x_{k}\right)\otimes \phi_{m} \right) \circ (\widetilde{\Delta}^{\langle m\rangle})^{(k)}.
\end{equation}
For any element $g \in\left(\mathbb{C}^{n-m}\right)^{\otimes k} \otimes Y(\mathfrak{g l}_{n})$, we define its components $g^{\boldsymbol{b}}$, $\boldsymbol{b}=(b_{1}, \ldots, b_{k})$, by the rule
\begin{equation}
g=\sum_{b_{1}, \ldots, b_{k}=m+1}^{n} \mathrm{\mathbf v}_{b_{1}-m} \otimes \cdots \otimes \mathrm{\mathbf v}_{b_{k}-m} \otimes g^{\boldsymbol{b}} .
\end{equation}
For any element $h \in \left(\mathbb{C}^{m}\right)^{\otimes k} \otimes Y(\mathfrak{g l}_{n}) $, we define its components $h^{\boldsymbol{a}}$, $\boldsymbol{a}=(a_{1}, \ldots, a_{k})$, by the rule
\begin{equation}
h=\sum_{a_{1}, \ldots, a_{k}=1}^{m} \mathbf{w}_{a_{1}} \otimes \cdots \otimes \mathbf{w}_{a_{k}} \otimes h^{\boldsymbol{a}} .
\end{equation}

Given nonnegative integers $\xi_1,\ldots,\xi_{n-1}$, and the variables
$\boldsymbol{t}=(t_1^1,\ldots, t^{n-1}_{\xi_{n-1}})$, see \eqref{tvar}, denote
\begin{equation}\begin{aligned}\label{dotnot}
&\boldsymbol{\xi}=(\xi_1,\ldots,\xi_{n-1}),\\
&\dot{\boldsymbol{\xi}}=(\xi_1,\ldots,\xi_{m-1}),\\
&\ddot{\boldsymbol{\xi}}=(\xi_{m+1},\ldots,\xi_{n-1}),\\
&\boldsymbol{t}=(t^1_1,\ldots,t^1_{\xi_1};t^2_1,\ldots,t^2_{\xi_2};t^{n-1}_1,\ldots,t^{n-1}_{\xi_n-1}),\\
& \dot{\boldsymbol{t}}=(t^1_1,\ldots,t^1_{\xi_1};t^2_1,\ldots,t^2_{\xi_2};t^{m-1}_1,\ldots,t^{m-1}_{\xi_m-1}),\\
&\ddot{\boldsymbol{t}} =(t^{m+1}_1,\ldots,t^{m+1}_{\xi_{m+1}};t^{m+2}_1,\ldots,t^{m+2}_{\xi_{m+2}};t^{n-1}_1,\ldots,t^{n-1}_{\xi_n-1}).
\end{aligned}
\end{equation}
Recall that $\xi^a=\xi_1+\xi_2+\ldots+\xi_a$, $a=1,\ldots,n-1$.

Let $V$ be a $\mathfrak{gl}_n$-module, $x\in \mathbb{C}$, and $V(x)$ be the evaluation $Y(\mathfrak{g l}_{n})$-module. Consider a weight singular vector $v\in V(x)$ of $\mathfrak{gl}_n$-weight $\boldsymbol{\Lambda}=(\Lambda_1,\ldots,\Lambda_n)$. That is, we have
\begin{equation} T^a_b(u)v=0,\; 1\leq b<a\leq n, \quad T_{a}^av=\dfrac{u-x+\Lambda_a}{u-x}v, \;a=1,\ldots,n. \end{equation}
We are interested in finding a formula for the vector $\mathbb{B}_{\boldsymbol{\xi}}(\boldsymbol{t})v$, where $\mathbb{B}_{\boldsymbol{\xi}}(\boldsymbol{t})$ is given by \eqref{bethevect}.

\begin{prop}
Let $v\in V(x)$ be a weight singular vector and $\,\boldsymbol\xi=(\xi_1,\ldots,\xi_{n-1})$
be a collection nonnegative integers. Then
\label{propsplitting}
\begin{equation}\label{splitting}
\mathbb{B}_{\boldsymbol{\xi}}(\boldsymbol{t})v =\
\sum_{\boldsymbol{a},\boldsymbol{b}} {\mathcal{T}} (\boldsymbol{t}^{m})^{\boldsymbol{a}}_{\boldsymbol{b}} \;\Bigl(\phi_m(\boldsymbol{t}^m)\left(\mathbb{B}_{\dot{\xi} }^{\langle m\rangle}(\dot{\boldsymbol{t}})\right)\Bigr)^{\boldsymbol{a}} \Bigl(\psi_m(\boldsymbol{t}^{m})\left(\mathbb{B}_ {\ddot{\xi}}^{\langle n-m\rangle}(\ddot{\boldsymbol{t}})\right)\Bigr)^{\boldsymbol{b}} v,
\end{equation}
where the sum is taken over all sequences $\boldsymbol{a}=(a_1,a_2,\ldots,a_{\xi_m})$, $\boldsymbol{b}=(b_1,b_2,\ldots,b_{\xi_m})$, such that $a_i\in \{1,2,\ldots, m\}$, $b_i\in \{m+1,m+2,\ldots,n\}$ for all $i=1,\ldots,\xi_m$, and
\begin{equation}\label{tcurve}
{\mathcal{T}}(\boldsymbol{t}^m)^{\boldsymbol{a}}_{\boldsymbol{b}}=T(t^m_1)^{a_1}_{b_1}T(t^m_2)^{a_2}_{b_2}\ldots T(t^m_{\xi_m})^{a_{\xi_m}}_{b_{\xi_m}}\,.
\end{equation}
\end{prop}

\begin{proof}
Using the definition of the maps $\psi_m(\boldsymbol{t}^{m})$ and $\phi_m(\boldsymbol{t}^{m})$, formula \eqref{splitting} can be written as
\begin{equation}\begin{aligned}
\label{Lemma1}
&\mathbb{B}_{\boldsymbol{\xi}}(\boldsymbol{t})v=
\sum_{\boldsymbol{a},\boldsymbol{b}} \, {\mathcal{T}}(\boldsymbol{t}^{ m})^{\boldsymbol{a}}_{\boldsymbol{b}} \;\left(\left(\overset{\rightarrow}{\prod\limits_{1< j\leq m} } \;\overset{\rightarrow}{\prod\limits_{1\leq i< j} }\;\ \overset{[j\; i]}{\mathbb{R}}(\boldsymbol{t}^j,\boldsymbol{t}^i)\right)
\;\overset{[m-1]}{\mathbb{T}}(\boldsymbol{t}^{m-1})\ldots \overset{[1]}{\mathbb{T}}(\boldsymbol{t}^{ 1})\right)^{\boldsymbol{\ell}_1}_{\boldsymbol{\ell}_2(\boldsymbol{a})} \times\\
&\times \left(\overset{[m+1]}{\mathbb{T}}(\boldsymbol{t}^{m+1})\ldots\overset{[n-1]}{\mathbb{T}}(\boldsymbol{t}^{n-1})
\left(\overset{\rightarrow}{\prod\limits_{m+1\leq j \leq n-1} }\;\;\overset{\rightarrow}{\prod\limits_{m\leq i<j} }\overset{[j\; i]}{\mathbb{R}}(\boldsymbol{t}^j,\boldsymbol{t}^i)\right) \right)^{\boldsymbol{\ell}_2(\boldsymbol{b})}_{ \boldsymbol{\ell}_3}v,
\end{aligned}\end{equation}
where the sequences $\boldsymbol{a},\boldsymbol{b}$ are as in \eqref{splitting} and
\begin{equation}\begin{aligned}
\boldsymbol{\ell}_1&=(\boldsymbol{1}^{\xi_1},\ldots,\boldsymbol{(m-1)}^{\xi_{m-1}},\boldsymbol{m}^{\xi_m},(\boldsymbol{m+1})^{\xi_m+1},\ldots,(\boldsymbol{n-1})^{\xi_n}),\\
\boldsymbol{\ell}_2(\boldsymbol{a})&=(\boldsymbol{2}^{\xi_1},\ldots,\boldsymbol{m}^{\xi_{m-1}},\boldsymbol{a},(\boldsymbol{m+1})^{\xi_m+1},\ldots,(\boldsymbol{n-1})^{\xi_n}),\\
\boldsymbol{\ell}_3&=(\boldsymbol{2}^{\xi_1},\ldots,\boldsymbol{m}^{\xi_{m-1}},\boldsymbol{(m+1)}^{\xi_m},(\boldsymbol{m+2})^{\xi_m+1},\ldots,\boldsymbol{n}^{\xi_n}).
\end{aligned}
\end{equation}
To prove formula \eqref{Lemma1}, we take the definition of $\mathbb{B}_{\boldsymbol{\xi}}(\boldsymbol{t})v$ following from formulas \eqref{boldt}, \eqref{bethevect}, and observe that using the Yang-Baxter equation, one can express $\mathbb{B}_{\boldsymbol{\xi}}(\boldsymbol{t})v$ as follows:
\begin{equation*}
\mathbb{B}_{\boldsymbol{\xi}}(\boldsymbol{t})v=\left(\left(\overset{\rightarrow}{\prod\limits_{1\leq i< j\leq m} }\overset{[j\; i]}{\mathbb{R}}\right)
\overset{[m]}{\mathbb{T}}\;\overset{[m-1]}{\mathbb{T}}\ldots \overset{[1]}{\mathbb{T}}\;\overset{[m+1]}{\mathbb{T}}\ldots
\overset{[n-1]}{\mathbb{T}}\left(\overset{\rightarrow}{\prod\limits_{m+1\leq j \leq n-1} }\;\overset{\rightarrow}{\prod\limits_{1\leq i<j} }\overset{[j\; i]}{\mathbb{R}}\right)\right)^{ \boldsymbol{\ell}_1}_{ \boldsymbol{\ell}_3}v.
\end{equation*}
In detail, that gives
\begin{equation}\begin{aligned}
\label{Lemma2}
\mathbb{B}_{\boldsymbol{\xi}}(\boldsymbol{t})v&=
\sum_{\boldsymbol{a},\boldsymbol{b},\boldsymbol{x},\boldsymbol{y},\boldsymbol{z}} \, {\mathcal{T}}(\boldsymbol{t}^{ m})^{\boldsymbol{a}}_{\boldsymbol{b}} \;\left(\overset{\rightarrow}{\prod\limits_{1< j\leq m} } \;\overset{\rightarrow}{\prod\limits_{1\leq i< j} }\;\ \overset{[j\; i]}{\mathbb{R}} (\boldsymbol{t}^j,\boldsymbol{t}^i)\right)^{\boldsymbol{\ell}_1}_{\boldsymbol{\ell}_4(\boldsymbol{y},\boldsymbol{a})}
\; {\mathcal{T}}(\boldsymbol{t}^{m-1})^{\boldsymbol{y}^{(m-1)}}_{\boldsymbol{x}^{(m-1)}} \ldots {\mathcal{T}}(\boldsymbol{t}^1)^{\boldsymbol{y}^{(1)}}_{\boldsymbol{x}^{(1)}} \times\\
&\times {\mathcal{T}}(\boldsymbol{t}^{m+1})^{\bullet}_{\boldsymbol{z}^{(m+1)}} \ldots{\mathcal{T}}(\boldsymbol{t}^{n-1})^{\bullet}_{\boldsymbol{z}^{(n-1)}}
\left(\overset{\rightarrow}{\prod\limits_{m+1\leq j \leq n-1} }\;\;\overset{\rightarrow}{\prod\limits_{1\leq i<j} }\overset{[j\; i]}{\mathbb{R}}(\boldsymbol{t}^j,\boldsymbol{t}^i)\right)^{\boldsymbol{\ell}_5(\boldsymbol{x},\boldsymbol{b},\boldsymbol{z})}_{ \boldsymbol{\ell}_3}v,
\end{aligned}\end{equation}
where

\begin{equation} \begin{aligned}
&\boldsymbol{x}=(\boldsymbol{x}^{(1)},\ldots,\boldsymbol{x}^{(m-1)}), \quad &&\boldsymbol{y}=(\boldsymbol{y}^{(1)},\ldots,\boldsymbol{y}^{(m-1)}) , \quad &&\boldsymbol{z}=(\boldsymbol{z}^{(m+1)},\ldots, \boldsymbol{z}^{(n-1)}),\\
&\boldsymbol{x}^{(l)}=(x^l_1,\ldots,x^l_{\xi_l}), &&\boldsymbol{y}^{(l)}=(y^l_1,\ldots,y^l_{\xi_l}), \quad &&\boldsymbol{z}^{(l)}=(z^l_1,\ldots,z^l_{\xi_l}), \end{aligned} \end{equation}
\begin{equation}
{\mathcal{T}}(\boldsymbol{t}^s)^{\boldsymbol{y^{(s)}}}_{\boldsymbol{x^{(s)}}}=T(t^s_1)^{y_1^s}_{x_1^s}\,T(t^s_2)^{y_2^s}_{x_2^s}\,\ldots \,T(t^s_{\xi_s})^{y_{\xi_s}^s}_{x_{\xi_s}^s},\quad\quad {\mathcal{T}}(\boldsymbol{t}^s)^{\bullet}_{\boldsymbol{z^{(s)}}}=T(t^s_1)^{s+1}_{z_1^s}\,T(t^s_2)^{s+1}_{z_2^s}\,\ldots \,T(t^s_{\xi_s})^{s+1}_{z_{\xi_s}^s},
\end{equation}
\begin{equation}\begin{aligned}
&\boldsymbol{\ell}_4(\boldsymbol{y},\boldsymbol{a})=\bigl(\boldsymbol{y}^{(1)} ,\ldots,\boldsymbol{y}^{(m-1)} ,\boldsymbol{a},(\boldsymbol{m+1})^{\xi_m+1},\ldots,(\boldsymbol{n-1})^{\xi_n}\bigr),\\
&\boldsymbol{\ell}_5(\boldsymbol{x},\boldsymbol{b},\boldsymbol{z})=\bigl(\boldsymbol{x}^{(1)} ,\ldots,\boldsymbol{x}^{(m-1)} ,\boldsymbol{b}, \boldsymbol{z}^{(m+1)},\ldots,\boldsymbol{z}^{(n-1)}\bigr)\;,
\end{aligned}\end{equation}
and the sum is taken over sequences $ \boldsymbol{a},\boldsymbol{b},\boldsymbol{x},\boldsymbol{y},\boldsymbol{z}$ with entries belonging to $\{1,\ldots, n\}$.

The next step is to show that for every $j=m+1,\ldots, n-1$, the product $\overset{\rightarrow}{\prod\limits_{1\leq i<j} }\overset{[j\; i]}{\mathbb{R}}$ in the second line of formula \eqref{Lemma2} can be truncated to $\overset{\rightarrow}{\prod\limits_{m\leq i<j} }\overset{[j\; i]}{\mathbb{R}}$ and the sum over $ \boldsymbol{x}=(\boldsymbol{x}^{(1)},\ldots,\boldsymbol{x}^{(m-1)})$ reduces to a single term with $\boldsymbol{x}=(\boldsymbol{2}^{\xi_1},\ldots,\boldsymbol{m}^{\xi_{m-1}})$. This can be done by induction on $j$. The key idea is to combine two observations. First, since $v$ is a weight singular vector, we have $z^s_k\geq s$ for all $s=m+1,\ldots, n-1$, and $ k=1,\ldots,\xi_s$. And second, the entries of $R$-matrix \eqref{rmatrix} have the property $R^{a\,b}_{c\,d}=\delta_{ac}\delta_{bd}$ if $b>c$, see formula \eqref{rentry}. We have worked out this reasoning in detail for the $\mathfrak{gl}_4$-case in \cite{KT}.

After the last step, formula \eqref{Lemma2} becomes
\begin{equation}\begin{aligned}\label{almost}
\mathbb{B}_{\boldsymbol{\xi}}(\boldsymbol{t})v&=
\sum_{\boldsymbol{a},\boldsymbol{b}, \boldsymbol{y},\boldsymbol{z}} \, {\mathcal{T}}(\boldsymbol{t}^{ m})^{\boldsymbol{a}}_{\boldsymbol{b}} \;\left(\overset{\rightarrow}{\prod\limits_{1< j\leq m} } \;\overset{\rightarrow}{\prod\limits_{1\leq i< j} }\;\ \overset{[j\; i]}{\mathbb{R}} (\boldsymbol{t}^j,\boldsymbol{t}^i)\right)^{\boldsymbol{\ell}_1}_{\boldsymbol{\ell}_4(\boldsymbol{y},\boldsymbol{a})}
\; {\mathcal{T}}(\boldsymbol{t}^{m-1})^{\boldsymbol{y}^{(m-1)}}_{\boldsymbol{2}^{\xi_{m-1}}} \ldots {\mathcal{T}}(\boldsymbol{t}^1)^{\boldsymbol{y}^{(1)}}_{\boldsymbol{2}^{ \xi_1}} \times\\
&\times {\mathcal{T}}(\boldsymbol{t}^{m+1})^{\bullet}_{\boldsymbol{z}^{(m+1)}} \ldots{\mathcal{T}}(\boldsymbol{t}^{n-1})^{\bullet}_{\boldsymbol{z}^{(n-1)}}
\left(\overset{\rightarrow}{\prod\limits_{m+1\leq j \leq n-1} }\;\;\overset{\rightarrow}{\prod\limits_{1\leq i<j} }\overset{[j\; i]}{\mathbb{R}}(\boldsymbol{t}^j,\boldsymbol{t}^i)\right)^{\boldsymbol{\ell}_6(\boldsymbol{b},\boldsymbol{z})}_{ \boldsymbol{\ell}_3}v,
\end{aligned}\end{equation}
where
\begin{equation}
\boldsymbol{\ell}_6(\boldsymbol{b},\boldsymbol{z}) =\bigl(\boldsymbol{2}^{\xi_1} ,\ldots,\boldsymbol{m}^{\xi_{m-1}} ,\boldsymbol{b}, \boldsymbol{z}^{(m+1)},\ldots,\boldsymbol{z}^{(n-1)}\bigr)\;.
\end{equation}
Formula \eqref{rentry} for entries of the $R$-matrix implies that the entry $\bigl(\overset{\rightarrow}{\prod\limits_{1< j\leq m} } \;\overset{\rightarrow}{\prod\limits_{1\leq i< j} }\;\ \overset{[j\; i]}{\mathbb{R}} \;\bigr)^{\boldsymbol{\ell}_1}_{\boldsymbol{\ell}_4(\boldsymbol{y},\boldsymbol{a})}$ equals zero unless $a_i\in \{1,2,\ldots, m\}$ for all $i=1,\ldots,\xi_m$, and the entry $\bigl(\overset{\rightarrow}{\prod\limits_{m+1\leq j \leq n-1} }\;\;\overset{\rightarrow}{\prod\limits_{1\leq i<j} }\overset{[j\; i]}{\mathbb{R}} \;\bigr)^{\boldsymbol{\ell}_6(\boldsymbol{b},\boldsymbol{z})}_{ \boldsymbol{\ell}_3}$ equals zero unless $b_i\in \{m+1,m+2,\ldots,n\}$ for all $i=1,\ldots,\xi_m$. Therefore, the sum over $\boldsymbol{a},\boldsymbol{b}$ in formula \eqref{almost} reduces to the same range of summation variables as in formula \eqref{Lemma1}. Furthermore, the sums over $\boldsymbol{y}$ and $\boldsymbol{z}$ in \eqref{almost} can be evaluated,
\begin{equation}\begin{aligned}
\sum_{ \boldsymbol{y} } \left(\overset{\rightarrow}{\prod\limits_{1< j\leq m} } \;\overset{\rightarrow}{\prod\limits_{1\leq i< j} }\;\ \overset{[j\; i]}{\mathbb{R}} (\boldsymbol{t}^j,\boldsymbol{t}^i)\right)^{\boldsymbol{\ell}_1}_{\boldsymbol{\ell}_4(\boldsymbol{y},\boldsymbol{a})}
\; {\mathcal{T}}(\boldsymbol{t}^{m-1})^{\boldsymbol{y}^{(m-1)}}_{\boldsymbol{m}^{\xi_{m-1}}} \ldots {\mathcal{T}}(\boldsymbol{t}^1)^{\boldsymbol{y}^{(1)}}_{\boldsymbol{2}^{ \xi_1}} =\\
\left(\left(\overset{\rightarrow}{\prod\limits_{1< j\leq m} } \;\overset{\rightarrow}{\prod\limits_{1\leq i< j} }\;\ \overset{[j\; i]}{\mathbb{R}}(\boldsymbol{t}^j,\boldsymbol{t}^i)\right)
\;\overset{[m-1]}{\mathbb{T}}(\boldsymbol{t}^{m-1})\ldots \overset{[1]}{\mathbb{T}}(\boldsymbol{t}^{ 1})\right)^{\boldsymbol{\ell}_1}_{\boldsymbol{\ell}_2(\boldsymbol{a})}
\end{aligned}\end{equation}
and
\begin{equation}
\begin{aligned}
\sum_{ \boldsymbol{z} } {\mathcal{T}}(\boldsymbol{t}^{m+1})^{\bullet}_{\boldsymbol{z}^{(m+1)}} \ldots{\mathcal{T}}(\boldsymbol{t}^{n-1})^{\bullet}_{\boldsymbol{z}^{(n-1)}}
\left(\overset{\rightarrow}{\prod\limits_{m+1\leq j \leq n-1} }\;\;\overset{\rightarrow}{\prod\limits_{1\leq i<j} }\overset{[j\; i]}{\mathbb{R}}(\boldsymbol{t}^j,\boldsymbol{t}^i)\right)^{\boldsymbol{\ell}_6( \boldsymbol{b},\boldsymbol{z})}_{ \boldsymbol{\ell}_3}= \\
\left(\overset{[m+1]}{\mathbb{T}}(\boldsymbol{t}^{m+1})\ldots\overset{[n-1]}{\mathbb{T}}(\boldsymbol{t}^{n-1})
\left(\overset{\rightarrow}{\prod\limits_{m+1\leq j \leq n-1} }\;\;\overset{\rightarrow}{\prod\limits_{m\leq i<j} }\overset{[j\; i]}{\mathbb{R}}(\boldsymbol{t}^j,\boldsymbol{t}^i)\right) \right)^{\boldsymbol{\ell}_2(\boldsymbol{b})}_{ \boldsymbol{\ell}_3}.
\end{aligned}
\end{equation}
Thus we transformed formula \eqref{almost} to formula \eqref{Lemma1}. Proposition \ref{propsplitting} is proved.
\end{proof}
\begin{remark} The statement of Proposition \ref{propsplitting} holds for any singular vector $v$ in any {$Y(\mathfrak{g l}_{n})$-module}.
\end{remark}

\begin{remark} For $\xi_m=0$, Proposition \ref{propsplitting} takes the form
\begin{equation}\label{xim}
\mathbb{B}_{\boldsymbol{\xi}}(\boldsymbol{t})v =
\phi_m \left(\mathbb{B}_{\boldsymbol{\dot{\xi}} }^{\langle m\rangle}(\dot{\boldsymbol{t}})\right) \; \psi_m\left(\mathbb{B}_ {\boldsymbol{\ddot{\xi}}}^{\langle n-m\rangle}(\ddot{\boldsymbol{t}})\right) v.
\end{equation}
Notice that the transformation of formula \eqref{almost} to formula \eqref{Lemma1} in the case $\xi_m=0$ remains nontrivial.
\end{remark}

\section{Weight functions}

Fix a positive integer $M$. Consider a collection of nonnegative integers $\boldsymbol{\lambda}=(\lambda_1,\ldots,\lambda_m)$ such that $ \sum_{k=1}^{m} \lambda_k=M$. Set $\lambda^{s}=\sum_{k=1}^s \lambda_k, \quad s=1, \ldots, m$. Introduce the variables $u_1^{s},\ldots, u_{\lambda^{s}}^{s}$, $s=1, \ldots, m$ . Denote $\boldsymbol{u}^{s}=\left(u_1^{s}, \ldots, u_{\lambda^{s}}^{s}\right)$ and $\boldsymbol{ {u}}=\left(\boldsymbol{u}^{1}, \ldots, \boldsymbol{u}^{m-1}\right)$. Let $\boldsymbol{J} = (J_1,\ldots, J_m)$ be a partition of $\{1,\ldots,M\}$ into disjoint subsets $J_1,\ldots, J_m$ such that $\abs{J_k}=\lambda_k$, $k=1,\ldots,m$.

Similarly, consider a collection of nonnegative integers $\boldsymbol{\mu}=(\mu_{m+1},\ldots,\mu_{n})$, such that $ \sum_{l=m+1}^n \mu_l~=M$. Set $\mu^{r}=\sum_{k=r+1}^n \mu_k, \quad r=m, \ldots, n-1$. Introduce the variables $v_1^{r},\ldots, v_{\mu^{r}}^{r}$, $r=m, \ldots, n-1$.
Denote $\boldsymbol{v}^{r}=\left(v_1^{r}, \ldots, v_{\mu^{r}}^{r}\right)$ and $\boldsymbol{v}=\left(\boldsymbol{v}^{m+1}, \ldots, \boldsymbol{v}^{n-1}\right)$. Let $\boldsymbol{I}=\left(I_{m+1},\ldots, I_{n}\right)$ be a partition of $\{1, \ldots, M\}$ into disjoint subsets $I_{m+1},\ldots, I_{n}$ such that $\left|I_j\right|=\mu_j$, $ j=m+1, \ldots, n$.

Given partitions $\boldsymbol{I},\boldsymbol{J}$ as above, we define rational functions $U_{\boldsymbol{I}}(\boldsymbol{v} ; \boldsymbol{v}^m )$ and $\widetilde{U}_{\boldsymbol{J}}(\boldsymbol{u} ; \boldsymbol{u}^m)$ that will be extensively used later in this paper. Set
\begin{equation}\label{Ufunction}
\begin{aligned}
&U_{\boldsymbol{I}}(\boldsymbol{v} ; \boldsymbol{v}^m)= \\
&=\prod_{l=m+1}^{n-1} \prod_{a=1}^{\mu^{l}}\left(\prod_{\substack{c=1 \\
i_c^{(l- 1)}=i_a^{(l)}}}^{\mu^{l-1}}\left(\dfrac{1}{v_a^{l}-v_c^{l-1}}\right) \prod_{\substack{d=1 \\
i_d^{(l-1)}>i_a^{(l)}}}^{\mu^{l-1}}\left(\dfrac{v_a^{l}-v_d^{l-1}+1}{v_a^{l}-v_d^{l-1}}\right) \prod_{b=a+1}^{\mu^{l}} \frac{v_b^{l}-v_a^{l}+1}{v_b^{l}-v_a^{l}}\right) ,
\end{aligned}
\end{equation}
where the numbers $i_c^{(l)}$ are defined as follows
\begin{equation}
\bigcup_{k=l+1}^n I_k=\{i_1^{(l)}<\ldots <i_{\mu^{l}}^{(l)}\}.
\end{equation}
Similarly, set
\begin{equation}\label{Utildefunction}\begin{aligned}
&\widetilde{U}_{\boldsymbol{J}}(\boldsymbol{u} ; \boldsymbol{u}^m)=\\
&=\prod_{l=2}^{m} \prod_{a=1}^{\lambda^{l}}\left(\prod_{\substack{c=1 \\ j_c^{(l-1)}<j_a^{(l)}}}^{\lambda^{l-1}}\left(\dfrac{u_a^{l}-u_c^{l-1}+1}{u_a^{l}-u_c^{l-1}}\right) \prod_{\substack{d=1 \\ j_d^{(l-1)}=j_a^{(l)}}}^{\lambda^{l-1}}\left(\dfrac{1}{u_a^{l}-u_d^{l-1}}\right) \prod_{b=a+1}^{\lambda^{l}} \frac{u_a^{l}-u_b^{l}+1}{u_a^{l}-u_b^{l}}\right),
\end{aligned}\end{equation}
where the numbers $j_c^{(l)}$ are defined as follows
\begin{equation}
\bigcup_{k=1}^l J_k=\{j_1^{(l)}<\ldots < j_{\lambda^{l}}^{(l)}\}.
\end{equation}

For a function $f\left(x_1, \ldots, x_k\right)$ of some variables, denote
\begin{equation*}
\operatorname{Sym}_{x_1, \ldots, x_k} f\left(x_1, \ldots, x_k\right)=\sum_{\sigma \in S_k} f\left(x_{\sigma(1)}, \ldots, x_{\sigma(k)}\right) .
\end{equation*}
Introduce the weight functions $W_{\boldsymbol{I}}(\boldsymbol{v} ; \boldsymbol{v}^m )$ and $\widetilde{W}_{\boldsymbol{J}}(\boldsymbol{u};\boldsymbol{u}^m )$
\begin{equation}\label{defW}
W_{\boldsymbol{I}}(\boldsymbol{v} ; \boldsymbol{v}^m )=\operatorname{Sym}_{v_1^{m+1}, \ldots, v_{\mu^{m+1}}^{m+1}} \ldots \operatorname{Sym}_{v_1^{n-1}, \ldots, v_{\mu^{n-1}}^{n-1}} U_{\boldsymbol{I}}(\boldsymbol{v} ; \boldsymbol{v}^m ),
\end{equation}
\begin{equation}\label{deftW}
\widetilde{W}_{\boldsymbol{J}}(\boldsymbol{u} ; \boldsymbol{u}^m )=\operatorname{Sym}_{u_1^{1}, \ldots, u_{\lambda^{1}}^{1}} \ldots \operatorname{Sym}_{u_1^{m-1}, \ldots, u_{\lambda^{m-1}}^{m-1}} U_{\boldsymbol{J}}(\boldsymbol{u} ; \boldsymbol{u}^m ),
\end{equation}

For $\sigma \in S_M$ and $\boldsymbol{J}=\left(J_1, \ldots, J_m\right)$, set $\sigma(\boldsymbol{J})=\left(\sigma\left(J_1\right), \ldots, \sigma\left(J_m\right)\right)$. Similarly, for $\boldsymbol{I}=(I_m,\ldots,I_n)$, $\sigma(\boldsymbol{I})=\left(\sigma\left(I_m\right), \ldots, \sigma\left(I_n\right)\right)$. For $a, b=1, \ldots, n$, let $s_{a, b}$ be the transposition of $a, b$.

Here is the main property of the weight functions, see for instance \cite{RTV2}.
\begin{lemm}\label{wcor}
Let $\boldsymbol{z}=(z_1,\ldots,z_{\xi_m})$. Then
\begin{equation*}
W_{\boldsymbol{I}}\left(\boldsymbol{v} ; z_1, \ldots, z_{a+1}, z_a, \ldots, z_{\xi_m} \right)\,=\,
\frac{z_{a+1}-z_{a}}{z_{a+1}-z_{a}-1}\,W_{s_{a, a+1}(\boldsymbol{I})}(\boldsymbol{v};\boldsymbol{z})
\,-\,\frac{1}{z_{a+1}-z_{a}-1}\,W_{\boldsymbol{I}}(\boldsymbol{v} ; \boldsymbol{z})\,,
\end{equation*}
\begin{equation*}
\widetilde{W}_{\boldsymbol{J}}\left(\boldsymbol{u};z_1,\ldots,z_{a+1},z_a, \ldots, z_{\xi_m}\right)\,=\,
\frac{z_{a}-z_{a+1}}{z_{a}-z_{a+1}-1}\,\widetilde{W}_{s_{a, a+1}(\boldsymbol{J})}(\boldsymbol{u} ; \boldsymbol{z})
\,-\,\frac{1}{z_{a}-z_{a+1}-1}\,\widetilde{W}_{\boldsymbol{J}}(\boldsymbol{u};\boldsymbol{z})\,.
\end{equation*}
\end{lemm}
This property provides us with the tool for the proof of the main result of this paper, Theorem \ref{mainth}.

\section {Main theorem}
The main result of this paper is Theorem $\ref{mainth}$ formulated at the end of this section. It will be approached in several steps.
We use the notation given in \eqref{dotnot}.

\begin{defin}\label{def}
For a collection $\boldsymbol{a}=(a_1,\ldots, a_{\xi_m})$ such that $a_i\in \{1,2,\ldots, m\}$ for all $i=1,\ldots,\xi_m$, we define a partition $\boldsymbol{J}(\boldsymbol{a})=(J_1,\ldots, J_m)$ of $\{1,\ldots,\xi_m\}$ by the rule $J_l=\{j| a_j=l\}$. Denote $|\cup_{r=1}^s J_r|=\zeta_s $, so that we have $\xi_m=\zeta_m\geq \zeta_{m-1}\geq\ldots\geq \zeta_1\geq 0$. Denote $\boldsymbol{\zeta}~=~(\zeta_{ 1},\ldots, \zeta_{m-1})$.

For a collection $\boldsymbol{b}=\left(b_1, \ldots, b_{\xi_m}\right)$ such that $b_i\in \{m+1,m+2,\ldots,n\}$ for all $i=1,\ldots,\xi_m$, we define a partition $\boldsymbol{I}(\boldsymbol{b})=(I_{m+1},\ldots, I_{n})$ of $\{1,\ldots,\xi_m\}$ by the rule $I_l=\{i| b_i=l\}$. Denote $|\cup_{r=s+1}^n I_r|=\eta_{s}$, so that we have $\xi_m=\eta_{m}\geq \eta_{m+1}\geq\ldots\geq \eta_{n-1}\geq 0$. Denote $\boldsymbol{\eta}=(\eta_{m+1},\ldots,\eta_{n-1})$.
\end{defin}
\begin{lemm}\label{bijection} The correspondences $\boldsymbol{a}\mapsto \boldsymbol{J}(\boldsymbol{a})$ and $\boldsymbol{b}\mapsto \boldsymbol{I}(\boldsymbol{b})$ are bijections.
\end{lemm}
\begin{proof} By inspection.
\end{proof}
Consider $\boldsymbol{\eta}\in \mathbb{Z}_{\geq 0}^{n-m-1}$, $\boldsymbol{\zeta}\in \mathbb{Z}_{\geq 0}^{m-1}$, such that $\eta_i\leq \xi_i$ and $ \zeta_j\leq \xi_j$ for all relevant $i,j$. Set
\begin{equation*}
\begin{aligned}
&\boldsymbol{t}^m=(t_1^m, \ldots, t_{\xi_m}^m)\\
&\ddot{\boldsymbol{t}}_{[ \boldsymbol{{\eta}}]}=\left(t_1^{m+1}, \ldots, t_{\eta_{m+1}}^{m+1} ; \ldots ; t_1^{n-1}, \ldots, t_{\eta_{n-1}}^{n-1}\right), \\
&\ddot{\boldsymbol{t}}_{(\boldsymbol{\eta}, \boldsymbol{\ddot{\xi}}]}=\left(t_{\eta_{m+1}+1}^{m+1}, \ldots, t_{\xi_{m+1}}^{m+1} ; \ldots ; t_{\eta_{n-1}+1}^{n-1}, \ldots, t_{\xi_{n-1}}^{n-1}\right) ,\\
&\dot{\boldsymbol{t}}_{[\boldsymbol{\dot{\xi}}-\boldsymbol{\zeta}]}=\left(t_1^1, \ldots, t_{\xi_1-\zeta_1}^1 ; \ldots ; t_1^{m-1}, \ldots, t_{\xi_{m-1}-\zeta_{m-1}}^{m-1}\right), \\
&\dot{\boldsymbol{t}}_{(\boldsymbol{\dot{\xi}}-\boldsymbol{\zeta}, \boldsymbol{\dot{\xi}}]}=\left(t_{\xi_1-\zeta_1+1}^1, \ldots, t_{\xi_1}^1 ; \ldots ; t_{\xi_{m-1}-\zeta_{m-1}+1}^{m-1}, \ldots, t_{\xi_{m-1}}^{m-1}\right) .
\end{aligned}
\end{equation*}

\begin{lemm}\label{psilemma}
For a given sequence $\boldsymbol{b}=(b_1,\ldots,b_{\xi_m})$, where $b_i\in \{m+1,m+2,\ldots,n\}$, consider the corresponding partition $\boldsymbol{I}(\boldsymbol{b})$ and the decreasing sequence $\eta_m=\xi_m \geq \eta_{m+1}\geq \cdots \geq \eta_{n-1}$ as described in Definition \ref{def}. Denote $\boldsymbol{\eta}=\left(\eta_{m+1} , \ldots, \eta_{n-1}\right) \in \mathbb{Z}_{\geq\ 0}^{n-m-1}$. Let $\tilde{v}$ be a vector, such that $T^a_c(u)\tilde{v}=0$, for $m\leq c<a\leq n$, and $T_{c}^c(u)\tilde{v}=(1+\Lambda^{c}(u-x)^{-1})\tilde{v}$, $c=m,\ldots,n$. Then for $\boldsymbol{\ddot{\xi}}-\boldsymbol{\eta}\in\mathbb{Z}^{n-m-1}_{\geq 0}$, we have
\begin{equation*}
\Bigl(\psi_{n-m}(\boldsymbol{t}^{m})\left(\mathbb{B}_{\boldsymbol{\ddot{\xi}}}^{\langle n-m\rangle}(\ddot{\boldsymbol{t}})\right)\Bigr)^{\boldsymbol{b}} \tilde v
\,=\prod_{b=m+1}^{n-1} \frac{1}{\left(\xi_b-\eta_b\right) !}\;{\operatorname{Sym}}_{\ddot{\boldsymbol{t}}}\left[U_{\boldsymbol{I}(\boldsymbol{b})}\left(\boldsymbol{\ddot{t}}_{[\boldsymbol{\eta}]}, \boldsymbol{t}^m\right) L_{\boldsymbol{\eta},\boldsymbol{\ddot{\xi}}}(\boldsymbol{\ddot{t}},\boldsymbol{t}^{ m}) \right]\tilde{v},
\end{equation*}
where $U_{\boldsymbol{I}(\boldsymbol{b})}\left(\boldsymbol{\ddot{t}}_{[\boldsymbol{\eta}]}, \boldsymbol{t}^m \right)$ is given by formula \eqref{Ufunction} and
\begin{equation}\label{Lfunct}
L_{\boldsymbol{\eta},\boldsymbol{\ddot\xi}}(\boldsymbol{\ddot{t}}\,)\,=
\prod_{a=m+1}^{n-2} \prod_{i=1}^{\eta_{a+1}} \prod_{j=\eta_a+1}^{\xi_a} \frac{t_i^{a+1}-t_j^a+1}{t_i^{a+1}-t_j^a} \prod_{l=m+1}^{n-1} \prod_{i=1}^{\eta_l} \frac{t_i^l-x+\Lambda^l}{t_i^l-x}
\;\psi_{n-m}\left(\mathbb{B}_{\boldsymbol{\ddot{\xi}}-\boldsymbol{\eta} }^{\langle n-m\rangle} \left(\boldsymbol{\ddot{t}}_{(\boldsymbol{\eta}, \boldsymbol{\ddot{\xi}}]}\right)\right) .
\end{equation}
If $\,\boldsymbol{\ddot{\xi}}-\boldsymbol{\eta}\not\in\mathbb{Z}^{n-m-1}_{\geq 0}$, then $\Bigl(\psi_{n-m}(\boldsymbol{t}^{m})\left(\mathbb{B}_ {\ddot{\boldsymbol{\xi}}}^{\langle n-m\rangle}(\ddot{\boldsymbol{t}})\right)\Bigr)^{\boldsymbol{b}} \tilde{v} =0$.
\end{lemm}
\begin{proof} The statement coincides with Lemma~4.3 in \cite{TV} up to a change of notation.
\end{proof}

\begin{lemm} \label{philemma}
For a given sequence ${\boldsymbol{a}}=(a_1,\ldots,a_{\xi_m})$, where $a_i\in \{1,2,\ldots, m\}$, consider the corresponding partition $\boldsymbol{J}(\boldsymbol{a})$ and the increasing sequence $\zeta_1 \leq \zeta_2\leq \cdots \leq \zeta_{m-1}\leq \zeta_{m}=\xi_m$ as described Definition \ref{def}. Denote $\boldsymbol{\zeta}=\left(\zeta_1, \ldots, \zeta_{m-1}\right) \in \mathbb{Z}_{\geq 0}^{m-1}$. Let $\hat{v}$ be a vector, such that $T^c_b(u)\hat{v}=0$, for $1\leq b<c\leq m$, and $T_{b}^b(u)\hat{v}=(1+\Lambda^{b}(u-x)^{-1})\hat{v}$, $b= 1,\ldots,m$. Then for $\boldsymbol{\dot{\xi}-\boldsymbol{\zeta}}\in\mathbb{Z}^{m-1}_{\geq0}$, we have
\begin{equation*}
\Bigl(\phi_m(\boldsymbol{t}^m)\left(\mathbb{B}_{\boldsymbol{\dot{\xi} }}^{\langle m\rangle}(\dot{\boldsymbol{t}})\right)\Bigr)^{\boldsymbol{a}} \hat v
\,=\,\prod_{l=1}^{m} \frac{1}{\left(\xi_l-\zeta_l\right) !}\;
{\operatorname{Sym}}_{\boldsymbol{\dot{t}}} \left[\widetilde{U}_{\boldsymbol{J}(\boldsymbol{a})}(\boldsymbol{\dot{t}}_{(\boldsymbol{\dot\xi}-\boldsymbol{\zeta},\boldsymbol{\dot\xi}]},\boldsymbol{t}^m) \widetilde{L}_{\boldsymbol{\zeta},\boldsymbol{\dot\xi}} (\boldsymbol{\dot{t}}\,)\right]\hat{v},
\end{equation*}
where $\widetilde{U}_{\boldsymbol{J}(\boldsymbol{a})}(\boldsymbol{\dot{t}}_{(\boldsymbol{\dot\xi}-\boldsymbol{\zeta},\boldsymbol{\dot\xi}]},{\boldsymbol{t}}^m )$ is given by formula \eqref{Utildefunction} and
\begin{equation}\label{Lbarfunction}\begin{aligned}
\widetilde{L}_{\boldsymbol{\zeta},\boldsymbol{\dot\xi}} (\boldsymbol{\dot{t}}\,)\,&{}=
\prod_{a=1}^{m-2}\,\prod_{i=1}^{\xi_{a+1}-\eta_{a+1}}\!\!\prod_{j=\xi_a-\eta_a+1}^{\xi_a} \frac{t_i^{a+1}-t_j^a+1}{t_i^{a+1}-t_j^a}\times{}\\
&{}\times\prod_{l=1}^{m-1}\,\prod_{i=0}^{\eta^l-1}\,\frac{t_{\xi_l-i}^l-x+\Lambda^{l+1}}{t_{\xi_l-i}^l-x}
\;\phi_m\left(\mathbb{B}_{\boldsymbol{\dot\xi}-\boldsymbol{\zeta}}^{\langle m\rangle}\left(\boldsymbol{\dot{t}}_{[\boldsymbol{\dot\xi}-\boldsymbol{\zeta}]}\right)\right) .
\end{aligned}\end{equation}
If $\boldsymbol{\dot{\xi}-\boldsymbol{\zeta}}\not\in\mathbb{Z}^{m-1}_{\geq0}$, then $\Bigl(\phi_m(\boldsymbol{t}^m)\left(\mathbb{B}_{\boldsymbol{\dot{\xi} }}^{\langle m\rangle}(\dot{\boldsymbol{t}})\right)\Bigr)^{\boldsymbol{a}} \hat{v}=0$.
\end{lemm}
\begin{proof}
The proof is similar to the proof of Lemma~4.3 in \cite{TV} with Corollary~3.4 there replaced by Corollary~3.2 {\it op.cit.}
\end{proof}

Consider collections $\boldsymbol{q}=(q_{sp}$ : $ s=m+1,\ldots, n, $ $p=1,\ldots, m$) of nonnegative integers, and introduce
\begin{equation}\begin{aligned}\label{seqq}
&\eta_k(\boldsymbol{q})= \sum_{s=k+1}^n\sum_{p=1}^m q_{sp}, \quad k=m+1,\ldots, n-1,\\
&\zeta_l(\boldsymbol{q})= \sum_{s=m+1}^n\sum_{p=1}^l q_{sp}, \quad l=1,\ldots,m-1.
\end{aligned}\end{equation}
Given $\boldsymbol{\xi}=(\xi_1,\ldots,\xi_{n-1})$, define a set
\begin{equation}\label{Qmn}
\mathcal{Q}_{m,n}=\Bigl\{\boldsymbol{q}=(q_{sp}): \;\sum_{s=m+1}^n\sum_{p=1}^m q_{sp}=\xi_m,\quad \eta_k(\boldsymbol{q})\leq \xi_k, \quad \zeta_l(\boldsymbol{q})\leq \xi_l, \; \text{ for all } k,l \Bigr\}.
\end{equation}
In addition, for any $\boldsymbol{q}\in \mathcal{Q}_{m,n}$, we consider the set $\mathcal{S}_{\boldsymbol{q}}$ of all pairs $(\boldsymbol{I},\boldsymbol{J})$ of partitions
$\boldsymbol{I}=(I_{m+1},\allowbreak\ldots, I_{n})$, $\,\boldsymbol{J}=(J_1,\ldots, J_m)$ of $\{1,\ldots,\xi_m\}$ with given cardinalities of intersections,
\begin{equation}\label{sq}
\mathcal{S}_{\boldsymbol{q}}=\bigl\{\;(\boldsymbol{I},\boldsymbol{J}):\;\; \abs{I_s\cap J_p}=q_{sp}, \;\;s=m+1,\ldots,n,\;p=1,\ldots,m\;\bigr\}.
\end{equation}
Notice that,
\begin{equation}
\abs{\cup_{r=k+1}^n I_r}=\eta_{k}(\boldsymbol{q}),\quad\abs{\cup_{r=1}^l J_r}=\zeta_l (\boldsymbol{q}).
\end{equation}
\begin{prop}\label{uprop1} Let $v\in V(x)$ be a weight singular vector of $\mathfrak{gl}_n$-weight $(\Lambda_1,\ldots\Lambda_n)$. Then
\begin{equation}\label{upropf}
\begin{aligned}
\mathbb{B}_{\boldsymbol{\xi}}(\boldsymbol{t})v\,={}\,&\prod_{l=1}^{\xi_m}\dfrac{1}{t_{l}^m-x}\,
\sum_{\boldsymbol{q}\in \mathcal{Q}_{m,n}}\,\prod_{i=m+1}^n\,\prod_{j=1}^m\,e^{q_{ij}}_{ij}\,
\prod_{a=1}^{m-1}\frac{1}{\left(\xi_a-\zeta_a\right) !}\;
\prod_{b=m+1}^{n-1} \frac{1}{\left(\xi_b-\eta_b\right)!}\times{}\\
&{}\times{}{\operatorname{Sym}}_{\boldsymbol{\ddot{t}}} \; { {\operatorname{Sym}_{\boldsymbol{\dot{t}}}}} \Biggl[L_{\boldsymbol{\eta},\boldsymbol{\ddot\xi}}(\boldsymbol{\ddot{t}}\,)\,\widetilde{L}_{\boldsymbol{\zeta},\boldsymbol{\dot\xi}} (\boldsymbol{\dot{t}}\,)\!\sum_{(\boldsymbol{I},\boldsymbol{J}) \in \mathcal{S}_{\boldsymbol{q}}}\!\widetilde{U}_{\boldsymbol{J}}(\boldsymbol{\dot{t}}_{(\boldsymbol{\dot\xi}-\boldsymbol{\zeta},\boldsymbol{\dot\xi}]},\boldsymbol{t}^m ) U_{\boldsymbol{I}}\left(\boldsymbol{\ddot{t}}_{[\boldsymbol{\eta}]}, \boldsymbol{t}^m \right) \Biggr]\,v,\!\!
\end{aligned}\end{equation}
where the sequences $\boldsymbol{\eta}=\left(\eta_{m+1}, \ldots, \eta_{n-1}\right)$ and $\boldsymbol{\zeta}=\left(\zeta_1, \ldots, \zeta_{m-1}\right)$, are given by the rule
\begin{equation}\eta_k=\eta_k(\boldsymbol{q}),\quad \zeta_l= \zeta_l(\boldsymbol{q}), \quad k=m+1,\ldots, n-1, \quad l=1,\ldots,m-1.\end{equation}
\end{prop}
\begin{proof}
Taking formula $\eqref{splitting}$ for $\mathbb{B}_{\boldsymbol{\xi}}(\boldsymbol{t})v$, we apply Lemma \ref{psilemma}
to the expression \newline
$\Bigl(\psi_m(\boldsymbol{t}^{m})\left(\mathbb{B}_{\boldsymbol{\ddot{\xi}}}^{\langle n-m\rangle}(\ddot{\boldsymbol{t}})\right)\Bigr)^{\boldsymbol{b}}v$. Next we observe that the vector $\psi_{n-m}\left(\mathbb{B}_{\boldsymbol{\ddot{\xi}}-\boldsymbol{\eta} }^{\langle n-m \rangle} \left(\boldsymbol{\ddot{t}}_{(\boldsymbol{\eta}, \boldsymbol{\ddot{\xi}}]}\right)\right)v$ in the right-hand side of \eqref{Lfunct} satisfies the conditions for the vector $\hat{v}$ in Lemma \ref{philemma}. Applying Lemma~\ref{philemma} to the expression $\Bigl(\phi_m(\boldsymbol{t}^m)\left(\mathbb{B}_{\boldsymbol{\dot{\xi} }}^{\langle m\rangle}(\dot{\boldsymbol{t}})\right)\Bigr)^{\boldsymbol{a}}\,\psi_{n-m}\left(\mathbb{B}_{\boldsymbol{\ddot{\xi}}-\boldsymbol{\eta} }^{\langle n-m \rangle} \bigl(\boldsymbol{\ddot{t}}_{(\boldsymbol{\eta}, \boldsymbol{\ddot{\xi}}]}\bigr)\right)v$, we obtain
\begin{equation}
\mathbb{B}_{\boldsymbol{\xi}}(\boldsymbol{t})v =
\sum_{\boldsymbol{a},\boldsymbol{b}} \Bigl( {\mathcal{T}} (\boldsymbol{t}^{m})\Bigr)^{\boldsymbol{a}}_{\boldsymbol{b}} \; {\operatorname{Sym}}_{\boldsymbol{\ddot{t}}} \; { {\operatorname{Sym}}_{\boldsymbol{\dot{t}}}} \Biggl[ \widetilde{U}_{\boldsymbol{J}(\boldsymbol{a})}(\boldsymbol{\dot{t}}_{(\boldsymbol{\dot\xi}-\boldsymbol{\zeta},\boldsymbol{\dot\xi}]},\boldsymbol{t}^m ) U_{\boldsymbol{I}(\boldsymbol{b})}\left(\boldsymbol{\ddot{t}}_{[\boldsymbol{\eta}]}, \boldsymbol{t}^m \right) L_{\boldsymbol{\eta},\boldsymbol{\ddot\xi}}(\boldsymbol{\ddot{t}}\,) \widetilde{L}_{\boldsymbol{\zeta},\boldsymbol{\dot\xi}} (\boldsymbol{\dot{t}}\,)\Biggr]v,
\end{equation}
where $\Bigl( {\mathcal{T}} (\boldsymbol{t}^{m})\Bigr)^{\boldsymbol{a}}_{\boldsymbol{b}} $ is given by \eqref{tcurve}, while the partitions $\boldsymbol{J}(\boldsymbol{a}), \boldsymbol{I}(\boldsymbol{b})$ and the sequences $\boldsymbol{\eta}=(\eta_{m+1}, \ldots,\eta_{n-1})$, $\boldsymbol{\zeta}=(\zeta_{1},\ldots,\zeta_{m-1})$ are described in Definition \ref{def}.

On the next step, we use the bijection between the sequences $\boldsymbol{a},\boldsymbol{b}$ and the partitions $\boldsymbol{I},\boldsymbol{J}$, see Lemma \ref{bijection}, to rewrite the sum $\sum_{\boldsymbol{a},\boldsymbol{b}}$ as the sum $\sum_{\boldsymbol{I},\boldsymbol{J}}$ over pairs of partitions. The latter sum can be further written as the double sum $\sum_{\boldsymbol{q}\in \mathcal{Q}_{m,n}}\sum_{ (\boldsymbol{I},\boldsymbol{J})\in\mathcal{S}_{\boldsymbol{q}}}$ over collections $\boldsymbol{q}\in \mathcal{Q}_{m,n}$ and pairs of partitions $(\boldsymbol{I},\boldsymbol{J})\in \mathcal{S}_{\boldsymbol{q}}$. Thus we obtain
\begin{equation}\begin{aligned}
\mathbb{B}_{\boldsymbol{\xi}}(\boldsymbol{t})v\,&{}=\!
\sum_{\boldsymbol{q}\in \mathcal{Q}_{m,n}}\;\sum_{(\boldsymbol{I},\boldsymbol{J}) \in \mathcal{S}_{\boldsymbol{q}}}\Bigl( {\mathcal{T}} (\boldsymbol{t}^{m})\Bigr)^{\boldsymbol{a}(\boldsymbol{J})}_{\boldsymbol{b}(\boldsymbol{I})}\times{}\\
&{}\times\,{\operatorname{Sym}}_{\boldsymbol{\ddot{t}}} \; { {\operatorname{Sym}}_{\boldsymbol{\dot{t}}}} \Biggl[ \widetilde{U}_{\boldsymbol{J}(\boldsymbol{a})}(\boldsymbol{\dot{t}}_{(\boldsymbol{\dot\xi}-\boldsymbol{\zeta},\boldsymbol{\dot\xi}]},\boldsymbol{t}^m ) U_{\boldsymbol{I}(\boldsymbol{b})}\left(\boldsymbol{\ddot{t}}_{[\boldsymbol{\eta}]}, \boldsymbol{t}^m \right) L_{\boldsymbol{\eta},\boldsymbol{\ddot\xi}}(\boldsymbol{\ddot{t}}\,) \widetilde{L}_{\boldsymbol{\zeta},\boldsymbol{\dot\xi}} (\boldsymbol{\dot{t}}\,)\Biggr]v\,,
\end{aligned}\end{equation}
where $\boldsymbol{a}(\boldsymbol{J})$, $\boldsymbol{b}(\boldsymbol{I})$ are the sequences corresponding to the partitions $\boldsymbol{I}$, $\boldsymbol{J}$.

Finally, we observe that in the module $V(x)$ we have $ \Bigl( {\mathcal{T}} (\boldsymbol{t}^{m})\Bigr)^{\boldsymbol{a}}_{\boldsymbol{b}} =\prod_{i,j} e^{q_{ij}}_{ji} \prod_{l=1}^{\xi_m}\dfrac{1}{t_{l}^m-x}$. The last expression and the product $L_{\boldsymbol{\eta},\boldsymbol{\ddot\xi}}(\boldsymbol{\ddot{t}}\,) \widetilde{L}_{\boldsymbol{\zeta},\boldsymbol{\dot\xi}} (\boldsymbol{\dot{t}}\,)$ depend only on the collection $\boldsymbol{q}$ and can be moved out of the sum $\sum_{ (\boldsymbol{I},\boldsymbol{J}) \in \mathcal{S}_{\boldsymbol{q}}}$. Proposition \ref{uprop1} is proved. \end{proof}

The next theorem is the main result of this paper. It will be proved in Section \ref{proofth}.

\begin{thrm}\label{mainth}
Let $v\in V(x)$ be a weight singular vector of $\mathfrak{gl}_n$-weight
$(\Lambda_1,\ldots\Lambda_n)$ and $\boldsymbol{\xi}=(\xi_1,\ldots,\xi_{n-1})$ be a collection
of nonnegative integers. Then
\begin{equation}\begin{aligned}\label{mainformula}
\!\mathbb{B}_{\boldsymbol{\xi}}(\boldsymbol{t})v\,={}\,&\prod_{l=1}^{\xi_m}\,\dfrac{1}{t_{l}^m-x}
\sum_{\boldsymbol{q}\in \mathcal{Q}_{m,n}}\,\prod_{i=m+1}^n\,\prod_{j=1}^m\,\dfrac{e^{q_{ij}}_{ij}}{q_{ij}!}\;\,
\prod_{a=1}^{m-1} \frac{1}{\left(\xi_a-\zeta_a\right) !}\;
\prod_{b=m+1}^{n-1} \frac{1}{\left(\xi_b-\eta_b\right)!}\times{}\\
&{}\times{{\operatorname{Sym}}_{{\boldsymbol{t}}^m}} \,{{\operatorname{Sym}}_{\dot{\boldsymbol{t}}}}\,{\operatorname{Sym}}_{\ddot{\boldsymbol{t}}} \left[ L_{\boldsymbol{\eta},\boldsymbol{\ddot\xi}}(\boldsymbol{\ddot{t}}\,) \widetilde{L}_{\boldsymbol{\zeta},\boldsymbol{\dot\xi}} (\boldsymbol{\dot{t}}\,) \; \widetilde{U}_{\boldsymbol{J}_0}(\boldsymbol{\dot{t}}_{(\boldsymbol{\dot\xi}-\boldsymbol{\zeta},\boldsymbol{\dot\xi}]},\check{\boldsymbol{t}}^m ) U_{\boldsymbol{\check{I}}_0}\left(\boldsymbol{\ddot{t}}_{[\boldsymbol{\eta}]}, \boldsymbol{t}^m \right) \Phi(\boldsymbol{t}^m)\right]v.
\end{aligned}\end{equation}
Here the set $\mathcal{Q}_{m,n}$ is given by \eqref{Qmn}, $\boldsymbol{I_0}=(I_{m+1},\ldots, I_{n})$ , $\boldsymbol{J_0}=(J_1,\ldots, J_m)$ is any pair of partitions of $\{1,\ldots,\xi_m\}$ such that $(\boldsymbol{I}_0,\boldsymbol{J}_0)\in \mathcal{S}_{\boldsymbol{q}}$, see \eqref{sq}, $\boldsymbol{\check{I}}_0=(\sigma_0(I_{m+1}),\ldots,\sigma_0 (I_{n}))$, where $\sigma_0\in S_{\xi_m}$ is the longest permutation, $\sigma_0(i)=\xi_m-i+1$, the sequences $\boldsymbol{\eta}=\left(\eta_{m+1}, \ldots, \eta_{n-1}\right)$, $\boldsymbol{\zeta}=\left(\zeta_1, \ldots, \zeta_{m-1}\right)$, are given by the rule $\eta_k=\eta_k(\boldsymbol{q}),\quad \zeta_l= \zeta_l(\boldsymbol{q})$ for all $k$, $l$, see \eqref{seqq}, the functions $ L_{\boldsymbol{\eta},\boldsymbol{\ddot\xi}}(\boldsymbol{\ddot{t}},\boldsymbol{t}^m )$, $\widetilde{L}_{\boldsymbol{\zeta},\boldsymbol{\dot\xi}} (\boldsymbol{\dot{t}},\boldsymbol{t}^m )$ are given by \eqref{Lfunct},\eqref{Lbarfunction}, the functions $U_{\boldsymbol{\check{I}}_0(\boldsymbol{b})}\left(\boldsymbol{\ddot{t}}_{[\boldsymbol{\eta}]}, \boldsymbol{t}^m \right)$, $\widetilde{U}_{\boldsymbol{J}_0(\boldsymbol{a})}(\boldsymbol{\dot{t}}_{(\boldsymbol{\dot\xi}-\boldsymbol{\zeta},\boldsymbol{\dot\xi}]},\check{\boldsymbol{t}}^m )$ are given by \eqref{Ufunction}, \eqref{Utildefunction}, $\check{\boldsymbol{t}}^m=(t_{\xi_m}^m,\ldots, t_{1}^m)$ and
\begin{equation}
\Phi(\boldsymbol{t}^m)= \prod_{1\leq a<b\leq \xi_m} \dfrac{t^m_a-t^m_b-1}{t^m_a-t^m_b} .
\end{equation}
\end{thrm}

\begin{remark} For $\xi_m=0$, we have $\boldsymbol{\eta}=(0,\ldots,0)$, $\boldsymbol{\zeta}=(0,\ldots,0)$. Then \, $\widetilde{U}_{\boldsymbol{J}_0(\boldsymbol{a})}(\boldsymbol{\dot{t}}_{(\boldsymbol{\dot\xi}-\boldsymbol{\zeta},\boldsymbol{\dot\xi}]},\check{\boldsymbol{t}}^m )=1,$ \break $U_{\boldsymbol{\check{I}}_0(\boldsymbol{b})}\left(\boldsymbol{\ddot{t}}_{[\boldsymbol{\eta}]}, \boldsymbol{t}^m \right)=1 $, $L_{\boldsymbol{\eta},\boldsymbol{\ddot\xi}}(\boldsymbol{\ddot{t}}\,)=\psi_{n-m}\left(\mathbb{B}_{\boldsymbol{\ddot{\xi}} }^{\langle n-m\rangle} (\boldsymbol{\ddot{t}} \,)\right)$, $\widetilde{L}_{\boldsymbol{\zeta},\boldsymbol{\dot\xi}} (\boldsymbol{\dot{t}}\,) =\phi_m\left(\mathbb{B}_{\boldsymbol{\dot\xi}}^{\langle m\rangle} (\boldsymbol{\dot{t}}\,)\right)$. Notice that $\psi_{n-m}\left(\mathbb{B}_{\boldsymbol{\ddot{\xi}} }^{\langle n-m\rangle} (\boldsymbol{\ddot{t}} \,)\right)$ is a symmetric function of $t^s_{1},\ldots, t_{\xi_s}^s$ for each $s=m+1,\ldots,n-1$, while $\phi_m\left(\mathbb{B}_{\boldsymbol{\dot\xi}}^{\langle m\rangle} (\boldsymbol{\dot{t}}\,)\right)$ is a symmetric function of $t^p_{1},\ldots, t_{\xi_p}^p$ for each $p=1,\ldots,m-1$. Therefore, formula \eqref{mainformula} for $\xi_m=0$ reduces to formula \eqref{xim},

\begin{equation}
\mathbb{B}_{\boldsymbol{\xi}}(\boldsymbol{t})v\,=\,
\phi_m \left(\mathbb{B}_{\boldsymbol{\dot{\xi}} }^{\langle m\rangle}(\dot{\boldsymbol{t}})\right)\psi_m\left(\mathbb{B}_ {\boldsymbol{\ddot{\xi}}}^{\langle n-m\rangle}(\ddot{\boldsymbol{t}})\right) v.
\end{equation}
\end{remark}

\begin{remark}
For $m=1$ after a minor simplification. formula \eqref{mainformula} reads
\begin{equation}
\mathbb{B}_{\boldsymbol{\xi}}(\boldsymbol{t})v\,=\,\prod_{l=1}^{\xi_1}\,\dfrac{1}{t^{1}_l-x}\;
\sum_{\boldsymbol q}\;\prod_{i=2}^n\,\dfrac{e^{q_i}_{i1}}{q_i!}\,\,
\prod_{b=2}^{n-1}\,\frac{1}{\left(\xi_b-\eta_b\right)!}\;{{\operatorname{Sym}}_{{\boldsymbol{t}}^1}}\,{\operatorname{Sym}}_{\ddot{\boldsymbol{t}}} \left[ L_{\boldsymbol{\eta},\boldsymbol{\ddot\xi}}(\boldsymbol{\ddot{t}})\,U_{\boldsymbol I}\bigl(\boldsymbol{\ddot{t}}_{[\boldsymbol{\eta}]}, \boldsymbol{t}^1\bigr)\Phi(\boldsymbol{t}^1)\right]v\,,
\end{equation}
where the sum is over $\boldsymbol q=(q_2,\ldots,q_n)\in\mathbb Z_{\geqslant0}^{n-1}$ such that $q_2+\dots+q_n=\xi_1$
and $\eta_b=q_{b+1}+\alb\dots\alb+q_n\leqslant\xi_b$ for all $b=2,\ldots,n-1$, while $\boldsymbol I=(I_2,\ldots, I_n)$
is any partition of $\{1,\ldots,\xi_1\}$ such that $|I_j|=q_j$ for all $j=2,\ldots,n$.
Thus for $m=1$, Theorem \ref{mainth} becomes Theorem~3.3 in \cite{TV}.

\smallskip
Similarly, for $m=n-1$ after a minor simplification, formula \eqref{mainformula} reads
\begin{equation}\begin{aligned}
\mathbb{B}_{\boldsymbol{\xi}}(\boldsymbol{t})v\,={}\,\,&\!\prod_{l=1}^{\xi_{n-1}}\dfrac{1}{t_{l}^{n-1}-x}
\;\sum_{\boldsymbol q}\;\prod_{j=1}^{n-1}\,\dfrac{e^{q_j}_{nj}}{q_j!}\,\,
\prod_{a=1}^{n-2}\,\frac{1}{\left(\xi_a-\zeta_a\right)!}\times{}\\
&{}\times{\operatorname{Sym}}_{{\boldsymbol{t}}^{n-1}}\,{{\operatorname{Sym}}_{\dot{\boldsymbol{t}}}}\, \left[ \widetilde{L}_{\boldsymbol{\zeta},\boldsymbol{\dot\xi}} (\boldsymbol{\dot{t}}\,)\,\widetilde{U}_{\boldsymbol J}\bigl(\boldsymbol{\dot{t}}_{(\boldsymbol{\dot\xi}-\boldsymbol{\zeta},\boldsymbol{\dot\xi}]},{\boldsymbol{t}}^{n-1}\bigr)\,\Phi(\boldsymbol{t}^{n-1})\right]v\,,
\end{aligned}\end{equation}
where the sum is over $\boldsymbol q=(q_1,\ldots,q_{n-1})\in\mathbb Z_{\geqslant0}^{n-1}$ such that $q_1+\dots+q_{n-1}=\xi_{n-1}$
and $\zeta_a=q_1+\alb\dots\alb+q_a\leqslant\xi_a$ for all $a=1,\ldots,n-2$, while $\boldsymbol J=(J_1,\ldots, J_{n-1})$
is any partition of $\{1,\ldots,\xi_{n-1}\}$ such that $|J_i|=q_i$ for all $i=1,\ldots,n-1$.
Thus for $m=n-1$, Theorem \ref{mainth} becomes Theorem~3.1 in \cite{TV}.
\end{remark}

\section{Proof of Theorem \ref{mainth}}\label{proofth}
Throughout this section, we use the notation in Theorem \ref{mainth}.
We will begin with two auxiliary lemmas.
\begin{lemm}\label{symlemma}
Consider functions $\,F(x_1,\ldots,x_{k})$ and $\,G(x_1,\ldots,x_l)$, $\,k\leq l$, such that
$\,G(x_1,\ldots,\allowbreak x_l)$ is a symmetric function of $x_1,\ldots,x_{k}$ and of $ x_{k+1},\ldots,x_l$.
Then
\begin{equation}
\begin{aligned}
{\operatorname{Sym}}_{x_1,\ldots,x_l} &\;\bigl[F(x_1,\ldots,x_{k}) G(x_1,\ldots,x_l)\bigr]=\\[4pt]
=&\dfrac{1}{k!}\,{\operatorname{Sym}}_{x_1,\ldots,x_k} \;\Bigl[\Bigl({\operatorname{Sym}}_{x_1,\ldots,x_{l} }F(x_1,\ldots,x_{k}) \Bigr)G(x_1,\ldots,x_l)\Bigr].
\end{aligned}\end{equation}
\end{lemm}
\begin{proof} By inspection.
\end{proof}

Recall the weight functions $W_{\boldsymbol{I}}(\boldsymbol{v} ; \boldsymbol{z})$, $\widetilde{W}_{\boldsymbol{J}}(\boldsymbol{u} ; \boldsymbol{z})$, see \eqref{defW},\eqref{deftW}, and the functions $L_{\boldsymbol{\eta},\boldsymbol{\ddot\xi}}(\boldsymbol{\ddot{t}},\boldsymbol{t}^m )$, $\widetilde{L}_{\boldsymbol{\zeta},\boldsymbol{\dot\xi}} (\boldsymbol{\dot{t}},\boldsymbol{t}^m )$ given by \eqref{Lfunct},\eqref{Lbarfunction}.
\begin{lemm}\label{ulem} One has
\begin{equation}
{\operatorname{Sym}}_{\boldsymbol{\ddot{t}}}\left[U_{\boldsymbol{I}}\left(\boldsymbol{\ddot{t}}_{[\boldsymbol{\eta}]}, \boldsymbol{t}^m \right) L_{\boldsymbol{\eta},\boldsymbol{\ddot\xi}}(\boldsymbol{\ddot{t}}\,) \right]
=\prod_{s=m+1}^{n-1} \dfrac{1}{\eta_s!}\;{\operatorname{Sym}}_{\boldsymbol{\ddot{t}}} \left[W_{\boldsymbol{I}}\left(\boldsymbol{\ddot{t}}_{[\boldsymbol{\eta}]}, \boldsymbol{t}^m \right) L_{\boldsymbol{\eta},\boldsymbol{\ddot\xi}}(\boldsymbol{\ddot{t}}\,) \right],
\end{equation}
\begin{equation}
{{\operatorname{Sym}}_{\boldsymbol{\dot{t}}}} \left[ \widetilde{U}_{\boldsymbol{J} }(\boldsymbol{\dot{t}}_{(\boldsymbol{\dot\xi}-\boldsymbol{\zeta},\boldsymbol{\dot\xi}]},\check{\boldsymbol{t}}^m ) \widetilde{L}_{\boldsymbol{\zeta},\boldsymbol{\dot\xi}} (\boldsymbol{\dot{t}}\,) \right]
=\prod_{s=1}^{m-1} \dfrac{1}{\zeta_s!} \;{{\operatorname{Sym}}_{\boldsymbol{\dot{t}}}} \left[ \widetilde{W}_{\boldsymbol{J} }(\boldsymbol{\dot{t}}_{(\boldsymbol{\dot\xi}-\boldsymbol{\zeta},\boldsymbol{\dot\xi}]},\check{\boldsymbol{t}}^m ) \widetilde{L}_{\boldsymbol{\zeta},\boldsymbol{\dot\xi}} (\boldsymbol{\dot{t}}\,)\right].
\end{equation}
\end{lemm}
\begin{proof}Observe that for every $s=m+1,\ldots,n-1$, the function $L_{\boldsymbol{\eta},\boldsymbol{\ddot\xi}}(\boldsymbol{\ddot{t}}\,)$ is
symmetric in $t^s_{1},\ldots, t_{\eta_s}^s$ and in $t_{\eta_s+1}^s,\ldots, t^s_{\xi_s}$. Similarly, for every $s=1,\ldots,m-1$, the function $\widetilde{L}_{\boldsymbol{\zeta},\boldsymbol{\dot\xi}} (\boldsymbol{\dot{t}}\,)$ is a symmetric in $t^s_{1},\ldots, t_{\xi_s-\zeta_s}^s$ and in $t_{\xi_s-\zeta_s+1}^s,\ldots, t^s_{\xi_s}$.
Then the statement follows from the formulas \eqref{defW},\eqref{deftW} by applying Lemma \ref{symlemma} several times.
\end{proof}

Recall that for $\sigma \in S_{\xi_m}$ and a partition $\boldsymbol{J}=\left(J_1, \ldots, J_m\right)$ of $\{1,\ldots,\xi_m\}$, we have $\sigma(\boldsymbol{J})=\left(\sigma\left(J_1\right), \ldots, \sigma\left(J_m\right)\right)$. Similarly, for $\boldsymbol{I}=(I_{m+1},\ldots,I_n)$, we have $\sigma(\boldsymbol{I})=\left(\sigma\left(I_{m+1}\right), \ldots, \sigma\left(I_{n}\right)\right)$.

The next proposition is the key point of the proof. For a collection $\boldsymbol{z}=(z_1,\ldots, z_{\xi_m})$ set
\begin{equation}
\Phi(\boldsymbol{z})= \prod_{1\leq a<b\leq \xi_m} \dfrac{z_a-z_b-1}{z_a-z_b} .
\end{equation}
\begin{prop}\label{wstatp} One has
\begin{equation}\begin{aligned}\label{wstat}
\sum_{\sigma\in S_{\xi_m}}W_{\sigma(\boldsymbol{I})}\left(\boldsymbol{\ddot{t}}_{[\boldsymbol{\eta}]}, \boldsymbol{z} \right) &\widetilde{W}_{\sigma(\boldsymbol{J})}(\boldsymbol{\dot{t}}_{(\boldsymbol{\dot\xi}-\boldsymbol{\zeta},\boldsymbol{\dot\xi}]},\boldsymbol{z} ) =\\&={ {\operatorname{Sym}}_{\boldsymbol{z} }}\Bigl[W_{\sigma_0(\boldsymbol{I})}(\boldsymbol{\ddot{t}}_{[\boldsymbol{\eta}]}, \boldsymbol{z})\widetilde{W}_{\boldsymbol{J}}(\boldsymbol{\dot{t}}_{(\boldsymbol{\dot\xi}-\boldsymbol{\zeta},\boldsymbol{\dot\xi}]}, {\boldsymbol{z}}^{\sigma_0}) \Phi(\boldsymbol{z})\Bigr] ,
\end{aligned}\end{equation}
where $\sigma_0$ is the longest permutation in $S_{\xi_m}$, $\sigma_0(i)=\xi_m-i+1$, and $\boldsymbol{z}^{\sigma_0}=(z_{\xi_m},\ldots,z_1)$.
\end{prop}
\begin{proof}Formula \eqref{wstat} can be written as
\begin{equation}\label{mainprop}
\sum_{\sigma\in S_{\xi_m}}W_{\sigma(\boldsymbol{I})}\left(\boldsymbol{\ddot{t}}_{[\boldsymbol{\eta}]}, \boldsymbol{z} \right) \widetilde{W}_{\sigma(\boldsymbol{J})}(\boldsymbol{\dot{t}}_{(\boldsymbol{\dot\xi}-\boldsymbol{\zeta},\boldsymbol{\dot\xi}]},\boldsymbol{z} ) =\sum_{\pi\in S_{\xi_m}} W_{\sigma_0(\boldsymbol{I})}(\boldsymbol{\ddot{t}}_{[\boldsymbol{\eta}]}, \boldsymbol{z}^\pi)\widetilde{W}_{\boldsymbol{J}}(\boldsymbol{\dot{t}}_{(\boldsymbol{\dot\xi}-\boldsymbol{\zeta},\boldsymbol{\dot\xi}]}, {\boldsymbol{z}}^{\pi\sigma_0}) \Phi(\boldsymbol{z}^\pi) ,
\end{equation}
$\boldsymbol{z}^\pi=(z_{\pi(1)},\ldots,z_{\pi(\xi_m)})$.
Then the proof of formula \eqref{mainprop} is literately the same as that of formula (7.16) in \cite{KT} with Corollary \ref{wcor} in this paper substituting Lemma~7.5 in \cite{KT}.\end{proof}

\begin{proof}[Proof of Theorem \ref{mainth}:]
First we apply Lemma \ref{ulem} to formula \eqref{upropf} and get
\begin{equation}
\begin{aligned}\label{mtp}
\mathbb{B}_{\boldsymbol{\xi}}(\boldsymbol{t})v\,=\prod_{l=1}^{\xi_m}\,\dfrac{1}{t_{l}^m-x}
\sum_{\boldsymbol{q}\in \mathcal{Q}_{m,n}}\,\prod_{i=m+1}^n\,\prod_{j=1}^m\, e^{q_{ij}}_{ij}\,
\prod_{a=1}^{m-1}\,\frac{1}{\left(\xi_a-\zeta_a\right)!\,\zeta_a!}\,\prod_{b=m+1}^{n-1} \frac{1}{\left(\xi_b-\eta_b\right)!\,\eta_b!}\times{}\!\!&\\
{}\times{\operatorname{Sym}}_{\boldsymbol{\ddot{t}}} \; {{\operatorname{Sym}_{\boldsymbol{\dot{t}}}}}\Biggl[ L_{\boldsymbol{\eta},\boldsymbol{\ddot\xi}}(\boldsymbol{\ddot{t}}\,) \widetilde{L}_{\boldsymbol{\zeta},\boldsymbol{\dot\xi}} (\boldsymbol{\dot{t}}\,)\!\sum_{(\boldsymbol{I},\boldsymbol{J}) \in \mathcal{S}_{\boldsymbol{q}}}\!\widetilde{W}_{\boldsymbol{J}}(\boldsymbol{\dot{t}}_{(\boldsymbol{\dot\xi}-\boldsymbol{\zeta},\boldsymbol{\dot\xi}]},\boldsymbol{t}^m ) W_{\boldsymbol{I} }\left(\boldsymbol{\ddot{t}}_{[\boldsymbol{\eta}]}, \boldsymbol{t}^m \right)\Biggr]v\,.\!\!\!&
\end{aligned}\end{equation}
Notice that every pair $(\boldsymbol{I},\boldsymbol{J}) \in \mathcal{S}_{\boldsymbol{q}}$ can be obtained from an arbitrary fixed pair $ (\boldsymbol{I}_0,\boldsymbol{J}_0) \in \mathcal{S}_{\boldsymbol{q}}$ by the action of the symmetric group $S_{\xi_m}$. Therefore, the inner sum in the right-hand side of formula \eqref{mtp} can be written in the following way,
\begin{equation}\label{twosums}\begin{aligned}
& \kern-.8em\sum_{(\boldsymbol{I},\boldsymbol{J}) \in \mathcal{S}_{\boldsymbol{q}}}W_{\boldsymbol{I}}\left(\boldsymbol{\ddot{t}}_{[\boldsymbol{\eta}]}, \boldsymbol{t}^m\right) \widetilde{W}_{\boldsymbol{J}}(\boldsymbol{\dot{t}}_{(\boldsymbol{\dot\xi}-\boldsymbol{\zeta},\boldsymbol{\dot\xi}]},\boldsymbol{t}^m)={}\\
&{}=\prod_{i=m+1}^n\,\prod_{j=1}^m\,\frac{1}{q_{ij}!}\,\sum_{\sigma\in S_{\xi_m}} W_{\sigma(\boldsymbol{I_0})}\left(\boldsymbol{\ddot{t}}_{[\boldsymbol{\eta}]}, \boldsymbol{t}^m\right) \widetilde{W}_{\sigma(\boldsymbol{J_0})}(\boldsymbol{\dot{t}}_{(\boldsymbol{\dot\xi}-\boldsymbol{\zeta},\boldsymbol{\dot\xi}]},\boldsymbol{t}^m)\,.
\end{aligned}\end{equation}
Here $\prod_{i,j} {(q_{i,j})!}$ is the cardinality of the isotropic subgroup of the pair $(\boldsymbol{I_0},\boldsymbol{J_0})$. Applying Proposition \ref{wstatp} for $\boldsymbol{z}=\boldsymbol{t}^{m}$ to the sum over permutations $\,\sum_{\sigma\in S_{\xi_m}}\!$ in the right-hand side of \eqref{twosums}, we get
\begin{equation}\begin{aligned}
\mathbb{B}_{\boldsymbol{\xi}}(\boldsymbol{t})v\,={}\,&\prod_{l=1}^{\xi_m}\dfrac{1}{t_{l}^m-x}
\sum_{\boldsymbol{q}\in\mathcal{Q}_{m,n}}\,\prod_{i=m+1}^n\,\prod_{j=1}^m\,\dfrac{e^{q_{ij}}_{ij}}{q_{ij}!}\,\,
\prod_{a=1}^{m-1}\,\frac{1}{\left(\xi_a-\zeta_a\right)!\,\zeta_a!}\,\prod_{b=m+1}^{n-1} \frac{1}{\left(\xi_b-\eta_b\right)!\,\eta_b!}\times{}\\
&{}\times{{\operatorname{Sym}}_{{\boldsymbol{t}}^m}} \,{{\operatorname{Sym}}_{\dot{\boldsymbol{t}}}}\,{\operatorname{Sym}}_{\ddot{\boldsymbol{t}}} \left[ L_{\boldsymbol{\eta},\boldsymbol{\ddot\xi}}(\boldsymbol{\ddot{t}}\,) \widetilde{L}_{\boldsymbol{\zeta},\boldsymbol{\dot\xi}} (\boldsymbol{\dot{t}}\,) \; \widetilde{W}_{\boldsymbol{J}_0}(\boldsymbol{\dot{t}}_{(\boldsymbol{\dot\xi}-\boldsymbol{\zeta},\boldsymbol{\dot\xi}]},\check{\boldsymbol{t}}^m ) W_{\boldsymbol{\check{I}}_0}\left(\boldsymbol{\ddot{t}}_{[\boldsymbol{\eta}]}, \boldsymbol{t}^m \right)
\Phi(\boldsymbol{t}^m) \right]v\,.
\end{aligned}\end{equation}
Finally, we use Lemma \ref{ulem} once more to transform the symmetrization ${{\operatorname{Sym}}_{\dot{\boldsymbol{t}}}}\,{\operatorname{Sym}}_{\ddot{\boldsymbol{t}}}$ in the last formula to the form in formula \eqref{mainformula}.
\end{proof}

\section*{Statements and Declarations: Conflict of interest / Data availability}

On behalf of all authors, the first author confirms that there is no conflict of interest. Additionally, data sharing is not applicable to this article, as no datasets were generated or analyzed during the current study.

\let\chapter\section


\begin{thebibliography}{MTV2}

\bibitem[FRV]{FRV}
G.\,Felder, R.\,Rimanyi, A.\,Varchenko,
Poincar\'e-Birkhoff-Witt expansions of the canonical elliptic differential form,
\href{https://doi.org/10.1090/conm/433}{\it Contemp. Math.} {\bf 433} (2007), 191--208.

\bibitem[FT]{FT} L.\,D.\,Faddeev, L.\,A.\!\;Takhtajan, The quantum method of the inverse problem and the Heisenberg $XYZ$ model, \href{https://iopscience.iop.org/article/10.1070/RM1979v034n05ABEH003909}{\it Russ. Math. Surv. } {\bf 34} (1979), no.\;5, 11--68.

\bibitem[K]{Kor} V.\,E.\,Korepin, Calculations of norms of Bethe wave function,
\href{https://doi.org/10.1007/BF01212176}{\textit{Commun. Math. Phys.}} \textbf{86} (1982),
no.\;3, 391--418.

\bibitem[KP]{KP}
S.\,Khoroshkin and S.\,Pakuliak,
Generating Series for Nested Bethe Vectors,
\href{https://doi.org/10.3842/SIGMA.2008.081}{\textit{SIGMA}} \textbf{4} (2008), 081, 1--23.

\bibitem[KPT]{KPT}
S.\,Khoroshkin, S.\,Pakuliak, V.\,Tarasov, Off-shell Bethe vectors and Drinfeld currents,
\href{https://www.sciencedirect.com/science/article/pii/S0393044007000265?via%3Dihub}{\it J.\;Geom. Phys.} {\bf 57} (2007), 1713--1732.

\bibitem[KR1]{KulRes82}
P.\,P.\,Kulish, N.\,Yu.\,Reshetikhin,
$GL(3)$-invariant solutions of the Yang--Baxter equation and associated quantum systems,
\href{https://doi.org/10.1007/BF01095104}{\it Zap. Nauch. Semin. LOMI\/} \textbf{120} (1982), 92--121, (in Russian);
{\it J.\;Sov. Math.}, \textbf{34} (1982), no.\;5, 1948--1971, (English translation).

\bibitem[KR2]{KulRes83}
P.\,P.\,Kulish and N.\,Yu.\,Reshetikhin, Diagonalisation of $GL(N)$ invariant transfer matrices
and quatum $N$-wave system (Lee model), \href{https://doi.org/10.1088\%2F0305-4470\%2F16\%2F16\%2F001}{\it J.\;Phys.\;A} \textbf{16} (1983), no.\;16, L591--L596.

\bibitem[KT]{KT} M.\,Kosmakov, V.\,Tarasov, New combinatorial formulae for nested Bethe vectors,
Preprint (2023), 1--27, \href{https://arxiv.org/abs/2312.00980}{\sf arXiv.2312.00980}.

\bibitem[M]{M}
A.\,Matsuo, An application of Aomoto--Gelfand hypergeometric functions to the $SU(n)$
Knizhnik-Zamolodchikov equation, \href{https://doi.org/10.1007/BF02102089}{\it Commun. Math. Phys.} {\bf 134} (1990), no.\;1, 65--77.

\bibitem[MV]{MV}
Y.\,Markov and A.\,Varchenko,
Hypergeometric solutions of trigonometric Knizhnik--Za\-molodchikov equations satisfy
dynamical difference equations, \href{https://www.sciencedirect.com/science/article/pii/S0001870801920274}{\it Adv. Math.} {\bf 166} (2002), no.\;1, 100--147.

\bibitem[MTT]{MTT}
T.\,Miwa, Y.\,Takeyama, and V.\,Tarasov, Determinant formula for solutions of
the quantum Knizhnik--Zamolodchikov equation associated with $U_q(\mathfrak{sl}_n)$,
\href{https://api.semanticscholar.org/CorpusID:15679297}{\it Publ.\;RIMS,}
\href{https://api.semanticscholar.org/CorpusID:15679297}{\it Kyoto}
\href{https://api.semanticscholar.org/CorpusID:15679297}{\it Univ.}
\textbf{35} (1999), no.\;6, 871--892.

\bibitem[MTV1]{MTV1}
E.\,Mukhin, V.\,Tarasov, A.\,Varchenko, Bethe eigenvectors of higher transfer matrices,
\href{http://dx.doi.org/10.1088/1742-5468/2006/08/P08002}{\textit{J.\;Stat. Mech. Theory Exp.}} (2006), no.\;8, P08002, 1--44.

\bibitem[MTV2]{MTV2}
E.\,Mukhin, V.\,Tarasov, A.\,Varchenko, Spaces of quasi-exponentials and representations of
the Yangian $Y(\mathfrak{gl}_n)$, \href{https://doi.org/10.1007/s00031-014-9275-8}
{\it Transform. Groups\/} \textbf{19} (2014), no.\;3, 861--885.

\bibitem[OPS]{OPS}
A.\,Oskin, S.\,Pakuliak, and A.\,Silantyev,
On the universal weight function for the quantum affine algebra $ U_q(\widehat{ \mathfrak{gl}}_N) $,
\href{https://doi.org/10.1090/S1061-0022-2010-01110-5}{\it St.\,Petersburg Math.\;J.} {\bf 21} (2010), no.\;4, 651--680.

\bibitem[RSV]{RSV}
R.\,Rim\'anyi, L.\,Stevens, A.\,Varchenko,
Combinatorics of rational functions and Poincare--Birchoff--Witt expansions of the canonical
$U(\mathfrak n_-)$-valued differential form, \href{https://www.sciencedirect.com/science/article/pii/S0001870801920274?via%3Dihub}{\it Ann. Comb.} {\bf 9}
(2005), no.\;1, 57--74.

\bibitem[RTV1]{RTV1}R.\,Rim\'anyi, V.\,Tarasov, A.\,Varchenko, Cohomology classes of conormal
bundles of Schubert varieties and Yangian weight functions,
\href{https://api.semanticscholar.org/CorpusID:119167231}{\it Math. Zeitschrift\/} \textbf{277} (2014),
no.\;3--4, 1085--1104.

\bibitem[RTV2]{RTV2}R.\,Rim\'anyi, V.\,Tarasov, A.\,Varchenko, Partial flag varieties, stable envelopes,
and weight functions. \href{https://ems.press/journals/qt/articles/13135}{\it Quantum Topol.\/} \textbf 6
(2015), no.\;2, 333--364.

\bibitem[RTV3]{RTV3}R.\,Rim\'anyi, V.\,Tarasov, A.\,Varchenko, Trigonometric weight functions as
$K$-theoretic stable envelope maps for the cotangent bundle of a flag variety, \href{https://api.semanticscholar.org/CorpusID:119138485}{\it J.\;Geom. Phys.}
\textbf{94} (2015), 81--119.

\bibitem[RTV4]{RTV4}R.\,Rim\'anyi, V.\,Tarasov, A.\,Varchenko, Elliptic and $K$-theoretic stable
envelopes and Newton polytopes, \href{https://doi.org/10.1007/s00029-019-0451-5}{\it Selecta Math.} {\bf 25} (2019), no.\;16, 1--43.

\bibitem[S1]{S1} N.\,A.\,Slavnov, Algebraic Bethe ansatz, Preprint (2018), 1--221,
\href{https://arxiv.org/abs/1804.07350}{\sf arXiv:1804.07350}.

\bibitem[S2]{S2} N.\,Slavnov, Algebraic Bethe ansatz and correlation functions ---
an advanced course, World Scientific, 2022, 380\;pp.

\bibitem[SV1]{SV1} V.\,Schechtman and A.\,Varchenko, Hypergeometric solutions of
Knizhnik--Zamolod\-chi\-kov equations, \href{https://doi.org/10.1007/BF00626523}{\it Lett. Math. Phys.} {\bf 20} (1990), no.\;4, 279--283.

\bibitem[SV2]{SV2} V.\,Schechtman and A.\,Varchenko, Arrangements of hyperplanes and Lie algebra
homology, \href{https://doi.org/10.1007/BF01243909}{\it Invent. Math.} {\bf 106} (1991), no.\;1, 139--194.

\bibitem[TV1]{TVC} V.\,Tarasov, A.\,Varchenko, Combinatorial formulae for nested Bethe vectors,
\href{https://www.emis.de/journals/SIGMA/2013/048/}{\it SIGMA} \textbf{9} (2013), 048, 1--28.

\bibitem[TV2]{TV} V.\,Tarasov, A.\,Varchenko, Jackson integral representations for solutions of
the Knizh\-nik--Zamolodchikov quantum equation, \textit{St.\,Petersburg Math.~J.} \textbf{6} (1994),
no.\;2, 275--313, \href{https://arxiv.org/abs/hep-th/9311040}{\sf arXiv:hep-th/9311040}.

\bibitem[TV3]{TV3} V.\,Tarasov, A.\,Varchenko, Selberg integrals associated with
$\mathfrak{sl}_3$, \href{https://doi.org/10.1023/B:MATH.0000010712.67685.9d}{\it Lett. Math. Phys.} {\bf 65} (2003), no.\;2, 173--185.

\bibitem[TV4]{TV4} V.\,Tarasov, A.\,Varchenko, Hypergeometric solutions of the quantum differential
equation of the cotangent bundle of a partial flag variety, \href{https://doi.org/10.1016/j.geomphys.2019.04.005}{\it Centr. Eur. J.\;Math.} {\bf 12} (2014),
no.\;5. 694--710.

\bibitem[TV5]{TV5} V.\,Tarasov, A.\,Varchenko, $q$-Hypergeometric solutions of quantum
differential equations, quantum Pieri rules, and Gamma theorem, \href{https://doi.org/10.1016/j.geomphys.2019.04.005}{\it J.\;Geom. Phys.}
\textbf{142} (2019), 179--212.

\end{thebibliography}
\end{document}